\documentclass{svjour2}                    
\smartqed  
\usepackage{graphicx}
\usepackage{amsmath}
\usepackage{amsfonts}
\usepackage{amssymb}
\usepackage{graphics}
\usepackage{graphicx}
\usepackage{amsbsy}
\usepackage{url}
\usepackage{enumerate}
\usepackage{verbatim}
\usepackage{fullpage}
\usepackage{multirow}
\usepackage{soul}
\usepackage{mathtools}
\usepackage{marginnote}

\usepackage{caption}
\usepackage{subcaption}
\usepackage{hyperref}
\usepackage{cite}

\newtheorem{ass}{Assumption}

\graphicspath{{../../leah_berlin/figures/}}
\usepackage{xcolor}

\newcommand{\norm}[1]{\left|#1\right|}
\newcommand{\Norm}[1]{\left\lVert#1\right\rVert}

\newcommand{\E}{{\mathbb E}}

\newcommand{\cH}{{\mathcal H}}
\newcommand{\cS}{{\mathcal S}}
\newcommand{\cG}{{\mathcal G}}
\newcommand{\cN}{{\mathcal N}}

\newcommand{\T}{{\mathbb T}}

\newcommand{\cA}{\mathcal A}
\newcommand{\cQ}{\mathcal Q}

\newcommand{\cL}{\mathcal L}
\newcommand{\clr}{L^2(\T^d ; \pi(x) dx)}

\DeclareMathOperator{\Tr}{Tr}

\begin{document}

\title{Variance Reduction using Nonreversible Langevin Samplers}

\titlerunning{Short form of title}        

\author{A. B. Duncan         \and
        T. Leli\`{e}vre \and  
        G. A. Pavliotis
}

\authorrunning{Duncan, Leli\'{e}vre, Pavliotis} 

\begingroup
\institute{A. B. Duncan \at
              Imperial College London, Department of Mathematics, South Kensington Campus, London SW7 2AZ,England
              \email{a.duncan@imperial.ac.uk}           
           \and
           T. Leli\`{e}vre \at
            CERMICS, Ecole des ponts, Université Paris-Est, 6-8 avenue Blaise Pascal, 77455 Marne la Vallée cedex 2, France
            \email{lelievre@cermics.enpc.fr}
            \and
           G. A. Pavliotis \at
             Imperial College London, Department of Mathematics, South Kensington Campus, London SW7 2AZ,England \email{g.pavliotis@imperial.ac.uk}}

\date{Received: date / Accepted: date}

\maketitle

\begin{abstract}
A standard approach to computing expectations with respect to a given target measure is to introduce an overdamped Langevin equation which is reversible with respect to the target distribution, and to approximate the expectation by a time-averaging estimator.  As has been noted in recent papers \cite{rey2014irreversible,lelievre2013optimal,wu2014attaining,hwang2014variance}, introducing an appropriately chosen nonreversible component to the dynamics is beneficial, both in terms of reducing the asymptotic variance and of speeding up convergence to the target distribution. In this paper we present a detailed study of the dependence of the asymptotic variance on the deviation from reversibility. Our theoretical findings are supported by numerical simulations.

\end{abstract}

%
%
\section{Introduction}
\label{sec:intro}

%
%

\subsection{Motivation}
In various applications arising in statistical mechanics, biochemistry, data science and machine learning \cite{lelievre2010free,gelman2014bayesian,mackay2003information,liu2008monte}, it is often necessary to compute expectations $$\pi(f) := \mathbb{E}_{\pi}f = \int_{\mathbb{R}^d} f(x)\pi(dx)$$ of an observable $f$ with respect to a target probability distribution $\pi(dx)$ on $\mathbb{R}^d$ with density $\pi(x)$ with respect to the Lebesgue measure, known up to the normalization constant\footnote{With a slight abuse of notation, we will denote by $\pi$ both the measure and the density}.  When the dimension $d$ is large, standard deterministic quadrature approaches become intractable, and one typically resorts to Markov-Chain Monte Carlo (MCMC) methods \cite{robert2013monte,gelman2014bayesian,liu2008monte}.  In this approach, $\pi(f)$ is approximated by a long-time average of the form:
\begin{equation}
\label{eq:estimator}
  \pi_T(f) := \frac{1}{T}\int_0^T f(X_t)\,dt,
\end{equation}
where $X_t$ is a Markov process chosen to be ergodic with respect to the target distribution $\pi$.  The Birkhoff-von Neumann Ergodic theorem \cite{krengel1985ergodic, revuz1999continuous,jacod1987limit} states that, for every observable $f \in L^1(\pi)$ we have
\begin{equation}
\label{eq:lln}
  \lim_{T\rightarrow \infty} \frac{1}{T}\int_0^T f(X_s)\,ds = \pi(f), \quad \pi-\mbox{ a.e. } X_0 = x.
\end{equation}
If $\pi$ possesses a smooth, strictly positive density, then a natural choice for $X_t$ is the overdamped Langevin dynamics
\begin{equation}
\label{eq:sde1}
  dX_t = \nabla \log \pi(X_t)\,dt + \sqrt{2}\,dW_t,
\end{equation}
where $W_t$ is a standard Brownian motion on $\mathbb{R}^d$.  Assuming that (\ref{eq:sde1}) possesses a unique strong solution which is non-explosive, the process $X_t$ is ergodic, with unique invariant distribution $\pi$, such that (\ref{eq:lln}) holds.  Under additional assumptions on the distribution $\pi$ and on the observable, this convergence result is accompanied by a central limit theorem which characterizes the asymptotic distribution of the fluctuations of $\pi_T(f)$ about $\pi(f)$, i.e.
\begin{equation}
\label{eq:clt}
  \sqrt{T}\left(\pi_T({f}) - \pi({f})\right) \xrightharpoonup{\mathcal{D}} \mathcal{N}(0, \sigma_f^2),
\end{equation}
where $\sigma_f^2$ is known as the asymptotic variance for the observable $f$.   For the reversible process (\ref{eq:sde1}) started from stationarity (i.e. $X_0 \sim \pi$), the Kipnis-Varadhan theorem \cite{kipnis1986central,cattiaux2012central} implies that \eqref{eq:clt} holds with asymptotic variance 
$$\sigma_f^2 = 2\langle f- \E_{\pi} f, (-\mathcal{S})^{-1} (f-\E_{\pi} f)\rangle_{\pi},$$ 
where $\langle \cdot, \cdot \rangle_\pi$ is the inner product in $L^2(\pi)$ and  $\mathcal{S}$ is the infinitesimal generator of $X_t$ defined by
\begin{equation}
\label{eq:generator_rev}
\mathcal{S} = \nabla \log \pi(x) \nabla \cdot + \Delta .
\end{equation}


Sampling methods based on Langevin diffusions have become increasingly popular due to their wide applicability and relative ease of implementation.  In practice, a discretisation of (\ref{eq:sde1}) may be used, which in general, will not be ergodic with respect to the target distribution $\pi$.  Thus, the  discretisation is augmented with a Metropolis-Hastings accept/reject step \cite{roberts1996exponential} which guarantees that the resulting Markov chain remains reversible with respect to $\pi$.  The resulting scheme is known as the Metropolis-Adjusted Langevin Algorithm (MALA), see \cite{roberts1996exponential}.  
\\\\
The computational cost required to approximate $\pi(f)$ using $\pi_T(f)$ to a given level of accuracy depends on the target distribution $\pi$, the observable $f$ and the process $X_t$ over which the long-time average is generated.  { Quite often, $X_t$ will exhibit some form of metastability \cite{schutte2013metastability,bovier2002metastability,bovier2005metastability,lelievre2013two}: the process $X_t$ will remain trapped for a long time exploring one mode,  with transitions between different modes occurring over longer timescales.  When the observable depends directly on the metastable degrees of freedom (i.e. the observable takes different values in different metastable regions)}, the asymptotic variance $\sigma_f^2$ of the estimator $\pi_T(f)$ may be very large.  As a result, more samples are required to obtain an estimate of $\pi(f)$ with the desired accuracy.  A similar scenario arises when the mass of $\pi$ is tightly concentrated along a low-dimensional submanifold of $\mathbb{R}^d$, as illustrated in Figure~\ref{fig:warped_trajectory}.  In this case, reversible dynamics, such as (\ref{eq:sde1}) will cause a very slow exploration of the support of $\pi$.  As a result, $\pi_T(f)$ will exhibit very large variance for observables which vary strongly along the manifold.
\\\\
\begin{figure}[htp]
\begin{center}
\includegraphics[width=0.6\textwidth]{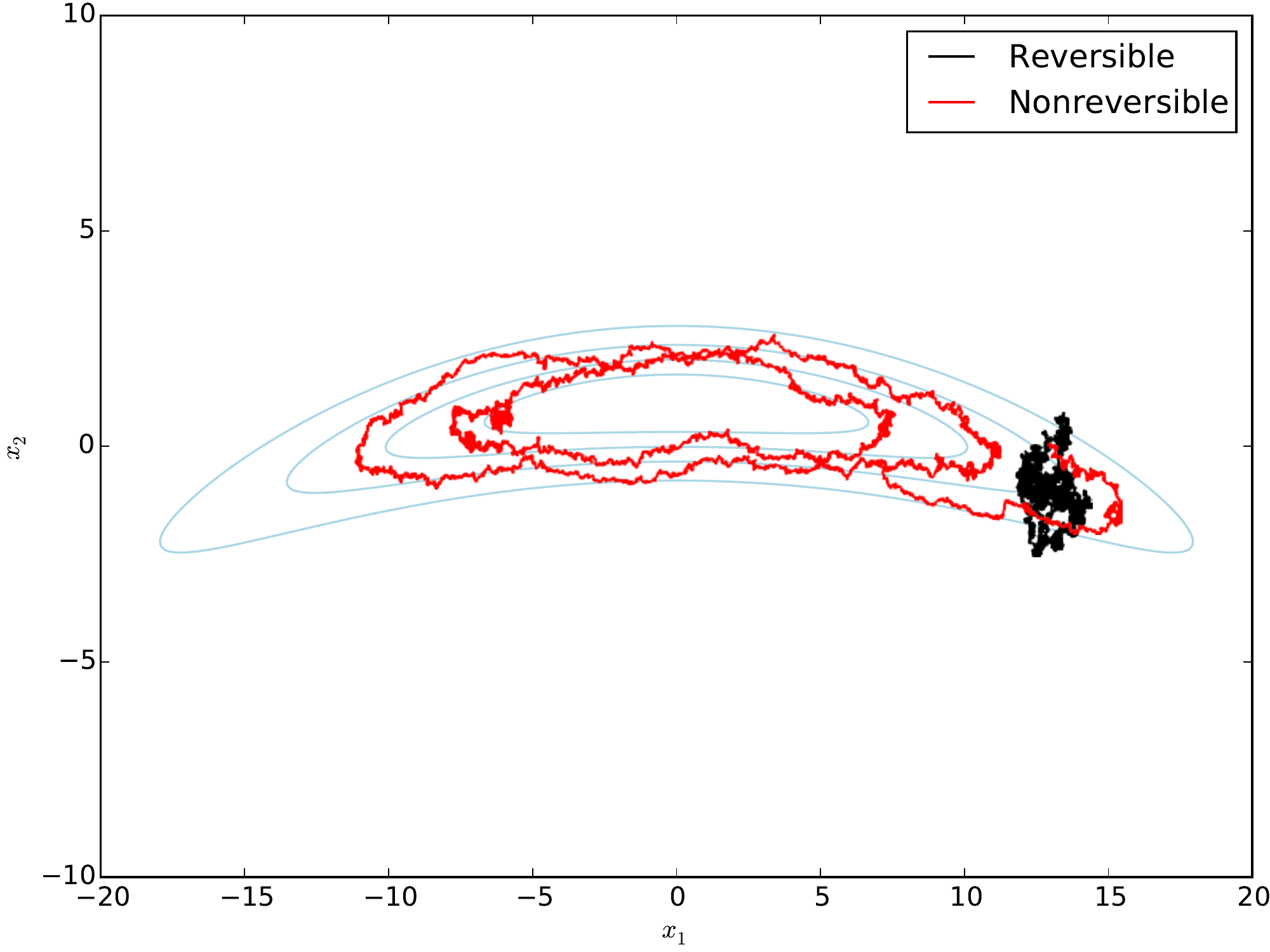}
\end{center}
\caption{Trajectories of a reversible MALA chain and a nonreversible Langevin sampler generating samples from a warped Gaussian distribution.  The mass of the distribution is contentrated along a one-dimensional submanifold.  Reversible samplers, such as MALA, perform a very slow exploration of the distribution, spending large amounts of time concentrated around a small region of the submanifold.  On the other hand, nonreversible samplers are able to perform long ``jumps'' along the level-curves of the distribution, thus able to explore the distribution far more effectively.}
\label{fig:warped_trajectory}
\end{figure}
As there are infinitely\footnote{Formally, all diffusion processes $X_t$ with drift $b(x)$ and diffusion $\sigma(x)$ and generator $\cL = b\cdot\nabla + \frac{1}{2}\sigma\sigma^\top:\nabla\nabla$ such that $\pi$ is the unique solution of the stationary Fokker-Planck equation $\cL^{\star} \pi =0$ can be used to sample from $\pi$.} many Markov processes with invariant distribution $\pi$, a natural question is whether such a process can be chosen to have optimal performance.  The two standard optimality criteria that are commonly used are: 
\begin{enumerate}[(a)]
\item With respect to speeding up convergence to the target distribution.
\item With respect to minimizing the asymptotic variance. 
\end{enumerate}
These criteria can be used in order to introduce a partial ordering in the set of Markov chains or diffusion processes that are ergodic with respect to a given probability distribution \cite{peskun1973optimum,Mira2001}.   From a practical perspective, the definite optimality criterion is that of minimizing the computational cost.  We address this issue (at least partially) later in this paper.
\\\\
Within the family of reversible samplers, much work has been done to derive samplers which exhibit good (if not optimal) computational performance.  This has motivated a number of variants of MALA which exploit geometric features of $\pi$ to explore the state space more effectively, including preconditioned MALA \cite{roberts2002langevin}, Riemannian Manifold MALA \cite{girolami2011riemann} and Stochastic Newton Methods \cite{martin2012stochastic}.  

\subsection{Nonreversible Langevin Dynamics}
An MCMC scheme which departs from the assumption of reversible dynamics is Hamiltonian MCMC \cite{neal2011mcmc}, which has proved successful in Bayesian inference.  By augmenting the state space with a momentum variable, proposals for the next step of the Markov chain are generated by following Hamiltonian dynamics over a large, fixed time interval.  The resulting nonreversible chain is able to make distant proposals.  Various methods have been proposed which are related to this general idea of breaking nonreversibility by introducing an additional dimension to the state space and introducing dynamics which explore the enlarged space while still preserving the equilibrium distribution.  In particular, the lifting method  \cite{turitsyn2011irreversible,hukushima2013irreversible,diaconis2000analysis} is one such method applied to discrete state systems, where the Markov chain is ``lifted'' from the state space $\Omega$ onto the space $\Omega\times \lbrace 1, -1 \rbrace$.  The transition probabilities in each copy are modified to introduce transitions between the copies to preserve the invariant distribution but now promote the sampler to generate long strides or trajectories.   Similar methods based on introducing nonreversibility into the dynamics of discrete state chains to speed up convergence have been applied with success in various applications \cite{neal2004improving,suwa2012general,sun2010improving,bierkens2014non,hennequin2014fast}.  These methods are also reminiscent of parallel tempering or replica exchange MCMC \cite{neal1996sampling}, which are aimed at efficiently sampling from multimodal distributions.
\\\\
It is well documented that breaking detailed balance, i.e. considering  a nonreversible  diffusion process that is ergodic with respect to $\pi$, can help accelerate convergence to equilibrium. In fact, it has been proved that, among all diffusion processes with additive noise that are ergodic with respect to $\pi$, the reversible dynamics~\eqref{eq:sde1} has the slowest rate of convergence to equilibrium, measured in terms of the spectral gap of the generator in $L^2(\mathbb{R}^d; \, \pi) = : L^2(\pi)$, c.f. \cite{lelievre2013optimal}. Adding a drift to~\eqref{eq:sde1} that is divergence-free with respect to $\pi$ and that preserves the invariant measure of the dynamics will always accelerate convergence to equilibrium~\cite{hwang1993accelerating,hwang2005accelerating,rey2014irreversible,lelievre2013optimal,wu2014attaining,ohzeki2013simple}. The optimal nonreversible perturbation can be identified and obtained in an algorithmic manner for diffusions with linear drift, whose invariant distribution is Gaussian \cite{lelievre2013optimal}; see also~\cite{wu2014attaining}. 
\\\\
The effect of nonreversibility in the dynamics on the asymptotic variance has also been studied. In \cite{rey2014irreversible} it is shown that small antisymmetric perturbations of the reversible dynamics always decrease the asymptotic variance, and more recently in \cite{rey2014variance} Friedlin-Wentzell theory is used to study the limit of infinitely strong antisymmetric perturbations.  In~\cite{hwang2014variance} the authors use spectral theory for selfadjoint operators to study the effect of antisymmetric perturbations on diffusions on both $\mathbb{R}^d$ and on compact manifolds and provide a general comparison result between reversible and nonreversible diffusions. This work is related with previous studies on the behavior of the spectral gap of the generator when the strength of the nonreversible perturbation is increased~\cite{franke2010behavior,constantin2008diffusion}. 
\\\\
\subsection{Objectives of this paper}
In this paper we present a detailed analytical and computational study of the effect of adding a nonreversible drift to the reversible overdamped Langevin dynamics (\ref{eq:sde1}) on the asymptotic variance $\sigma^2_f$. We will consider the nonreversible dynamics defined by \cite{hwang1993accelerating,hwang2005accelerating} :
\begin{equation}
\label{e:nonreversible} dX^{\gamma}_{t} = \big(\nabla
\log \pi(X^{\gamma}_{t}) + \alpha \gamma(X^{\gamma}_{t}) \big) \, dt + \sqrt{2 } \, dW_{t},
\end{equation} 
where the smooth vector field $\gamma$ is taken to be divergence-free with respect to the distribution $\pi$,
\begin{equation}\label{e:divergence_free} \nabla \cdot \big( \gamma  \pi
\big) = 0.
\end{equation}
There are infinitely many vector fields that satisfy~\eqref{e:divergence_free} and a general formula can be derived using Poincar\'{e}'s lemma~\cite{hwang2005accelerating,rey2014variance}. We can, for example, construct such vector fields by taking  
\begin{equation}\label{e:perturb-gen}
\gamma (x)=J \nabla \log \pi(x), \quad J = - J^{T}.
\end{equation} 
In~\eqref{e:nonreversible} we have already normalized the various vector fields and we have introduced a parameter $\alpha$ which measures the strength of the deviation from reversible dynamics. The generator of the diffusion process $X_t^{\gamma}$ can be decomposed into a symmetric and an antisymmetric part in $L^2(\pi)$, representing the reversible and irreversible parts of the dynamics, respectively~\cite[Sec. 4.6]{pavliotis2014stochastic}:
\begin{equation}
\label{eq:operator_decomposition}
\mathcal{L} =  \mathcal{S} + \alpha\mathcal{A},
\end{equation}
where $\mathcal{S}$ is given in (\ref{eq:generator_rev}) and $\mathcal{A} = \gamma \cdot\nabla$.  We are particularly interested in quantifying the reduction in the asymptotic variance obtained caused by breaking detailed balance. It is well known~\cite{komorowski2012fluctuations,kipnis1986central}  that for square integrable observables, the asymptotic variance, denoted by  $\sigma^2_f(\alpha)$ can be written in terms of the solution of the Poisson equation
\begin{equation}
\label{eq:poisson}
  -( \mathcal{S} + \alpha\mathcal{A})\phi = f- \E_{\pi}f,
\end{equation}
as
\begin{equation}
\label{eq:asymptotic_var-intro}
\sigma_f^2 (\alpha) =  2\int \phi(- \cL) \phi \,\pi(dx).
\end{equation}

Our first goal is to obtain an explicit formula for $\sigma^2_f(\alpha)$ in terms of $\mathcal{S}$ and $\mathcal{A}$.   Moreover, through a spectral decomposition of the operator $(-\mathcal{S})^{-1}\mathcal{A}$ we investigate the dependence of the asymptotic variance on the strength of the nonreversible perturbation. Based on earlier work by Golden and Papanicolaou \cite{golden1983bounds}, Majda et al \cite{majda1993effect,avellaneda1991integral} and Bhattacharya et al \cite{bhattacharya1999multiscale,bhattacharya1989asymptotics}, in \cite{pavliotis2010asymptotic}, an expression is obtained for the effective self-diffusion coefficient for a tagged particle whose microscopic behaviour is determined by a nonreversible Markov process. The basic idea is the introduction of the operator 
\begin{equation} 
\label{eq:operator_G}
\mathcal{G} := (-\mathcal{S})^{-1} \mathcal{A}, 
\end{equation} 
which is skewadjoint in the weighted Sobolev space 
$$\cH^1 = \lbrace f \in L^2(\pi) \,: \pi(f) = 0 \mbox{ and } \langle (-\cS)f, f\rangle_{\pi} < \infty \rbrace,$$
where $\langle \cdot, \cdot\rangle_{\pi}$ denotes the $L^2(\pi)$ inner product;  see Section \ref{sec:var_anal}.  We follow a similar approach in this paper, from which a quite explicit formula for the asymptotic variance $\sigma_f^2(\alpha)$ is derived using the spectral theorem for selfadjoint operators. From this expression we can immediately deduce that the addition of a nonreversible perturbation reduces the asymptotic variance and we can also study the small and large $\alpha$ asymptotics of $\sigma_f^2$; see Theorem~\ref{thm:deff_expansion}.  One of the results that we prove is that the large $\alpha$ behaviour of the asymptotic variance depends on the detailed properties of the vector field $\gamma$: when the nullspace of the antisymmetric part of the generator $\cA = \gamma\cdot\nabla$ consists of only constants in $\cH^1$, then the asymptotic variance converges to $0$.  {Indeed, in this case, in the $|\alpha|\rightarrow \infty$ limit, the limiting dynamics become deterministic, characterized by the Liouville operator $\gamma\cdot\nabla$. On the other hand, when the nullspace of $\cA$ contains nontrivial functions in $\cH^1$ there will exist observables for which the asymptotic variance $\sigma_f^2$ converges to a positive constant as $|\alpha| \rightarrow \infty$}.
\\\\
The effect of the antisymmetric part on the long time dynamics of diffusion processes has been also  studied extensively in the context of turbulent diffusion \cite{majda1999simplified} and fluid mixing. The effect of an incompressible flow on the convergence of the solution to the advection--diffusion equation on a compact manifold to its mean value (i.e. when $\pi \equiv 1$) was first studied in~\cite{constantin2008diffusion}. In particular, the concept of a relaxation enhancing flow was introduced and it was shown that a divergence-free flow is relaxation enhancing if and only if the Liouville operator $\cA = v \cdot \nabla$ has no eigenfunctions in the Sobolev space $H^{1}$. An equivalent formulation of this result is that an incompressible flow is relaxation enhancing if and only if it is weakly mixing. Examples of relaxation enhancing flows are given in~\cite{constantin2008diffusion}. This problem was studied further in~\cite{franke2010behavior}, where it also mentioned that there are very few examples of relaxation enhancing flows. In \cite{franke2010behavior} it is shown that the spectral gap of the advection-diffusion operator $\cL = \alpha v \cdot \nabla + \Delta$, for a  divergence-free drift $v$ and $\alpha \in \mathbb{R}$, remains bounded above in the limit as $\alpha \rightarrow \pm\infty$ by a negative constant if and only if the advection operator has an eigenfunction in $H^{1}$.  The results are reminiscent of the results we mention above on the necessary and sufficient condition to obtain a reduction of asymptotic variance in the limit $\alpha\rightarrow \pm \infty$.
\\\\
Our analysis of the asymptotic variance $\sigma^2_f$, based on the careful study of the Poisson equation~\eqref{eq:poisson}, enables us to study in detail the problem of finding the nonreversible perturbation giving rise to minimum asymptotic variance for diffusions with linear drift, i.e. diffusions whose invariant measure is Gaussian, over a large class of observables. Diffusions with linear drift were considered in \cite{lelievre2013optimal} where the optimal nonreversible perturbation with respect to accelerating convergence was obtained.  For linear and quadratic observables, we can give a complete solution to this problem, and construct nonreversible perturbations that provide a dramatic reduction in asymptotic variance.  Moreover, we demonstrate that the conditions under which the variance is reduced are very different from those of maximising the spectral gap discussed in \cite{lelievre2013optimal}. In particular, we show how a nonreversible perturbation can dramatically reduce the asymptotic variance for the estimator $\pi_T(f)$, even though no such improvement can be made on the rate of convergence to equilibrium.
\\\\
Guided by our theoretical results, we can then study numerically the reduction in the asymptotic variance due to the addition of a nonreversible drift for some toy models from molecular dynamics. In particular, we study the problem of computing expectations of observables with respect to a warped Gaussian \cite{haario1999adaptive} in two dimensions, as well as a simple model for a dimer in a solvent~\cite{lelievre2010free}. The numerical experiments reported in this paper illustrate that a judicious choice of the nonreversible perturbation, dependent on the target distribution and the observable,  can dramatically reduce the asymptotic variance.
\\\\
To compute $\pi_T(f)$ numerically, we use an Euler-Maruyama discretisation of (\ref{e:nonreversible}).  The resulting discretisation error introduces an additional bias in the estimator for $\pi(f)$, see \cite{mattingly2010convergence} for a comprehensive error analysis.  This imposes additional constraints on the magnitude of the nonreversible drift, since increasing $\alpha$ arbitrarily will give rise to a large discretisation error which must be controlled by taking smaller timesteps.  A natural question is whether the increase in the computational cost due to the necessity of taking smaller timesteps negates any benefits of the resulting variance reduction. To study this problem, we compare the computational cost of the unadjusted nonreversible Langevin sampler with the corresponding MALA scheme.\footnote{As we illustrate in our paper, there is no point in considering the Metropolis adjusted sampler with a nonreversible proposal, since the addition of the accept-reject step renders the resulting Markov chain reversible and any nonreversibility--induced variance reduction is lost.} Our numerical results, together with the theoretical analysis for diffusions with linear drift, show that the nonreversible Langevin sampler can outperform the MALA algorithm, provided that the nonreversible perturbation is well--chosen.  Finally, we consider a higher order numerical scheme for generating samples of (\ref{e:nonreversible}), based on splitting the reversible and nonreversible dynamics. Numerically, we investigate the properties of this integrator, and demonstrate that its improved stability and discretisation error make it a good numerical scheme for computing the estimator $\pi_T(f)$ using a nonreversible diffusion.
\\\\
The rest of the paper is organized as follows.  In Section \ref{sec:clt} we describe how the central limit theorem (\ref{eq:clt}) arises from the solution of the Poisson equation associated with the generator of the dynamics.  In Section ~\ref{sec:var_anal} we  analyse the asymptotic variance and formulate the problem of finding the optimal perturbation with respect to minimising $\sigma^2_f$, for a fixed observable, or over the space of square-integrable observables.  Moreover, following the analysis described in \cite{pavliotis2010asymptotic}, we derive a spectral representation for the asymptotic variance $\sigma^2_f$, in terms of the discrete spectrum of the operator $\mathcal{G}$ defined in (\ref{eq:operator_G}), and recover estimates for the asymptotic variance for any value of $\alpha$.  In Section \ref{sec:gauss cxx}, we consider the case of Gaussian diffusions, which are amenable to explicit calculation to demonstrate the theory presented in this paper.  In Section \ref{sec:numerics}, we provide various numerical examples to complement the theoretical results.  Finally, in Section \ref{sec:comp_cost} we describe the bias-variance tradeoff for nonreversible Langevin samplers, and explore their computational cost.  Conclusions and discussion on further work are presented in Section \ref{sec:conclusions}.

\section{The Central Limit Theorem and estimates on the asymptotic variance via the Poisson equation}
\label{sec:clt}
In this section we make explicit some sufficient conditions under which the estimator 
$$\pi_T(f) = \frac{1}{T}\int_0^T f(X_s^\gamma)\,ds, $$
where $X_t^\gamma$ denotes the solution of (\ref{e:nonreversible}), satisfies a central limit theorem of the form (\ref{eq:clt}).  The fundamental ingredient for proving such a central limit theorem is establishing the well-posedness of the Poisson equation
\begin{equation}
\label{eq:poisson1}
  -\mathcal{L}\phi(x) = f(x) - \pi({f}),\quad \pi(\phi) = 0,
\end{equation}
for all bounded and continous functions $f:\mathbb{R}^d\rightarrow \mathbb{R}$, where $\cL$ is defined by (\ref{eq:operator_decomposition}), and obtaining estimates on the growth of the unique solution $\phi$.  As described previously, we shall assume that $\pi$ possesses a smooth, strictly positive density, also denoted by $\pi(x)$, such that $\int_{\mathbb{R}^d} \pi(x)\,dx < \infty$ and that the SDE (\ref{e:nonreversible}) has a unique strong solution for all $\alpha \in \mathbb{R}$.   Applying the results detailed in \cite{glynn1996liapounov,meyn1993survey}, we shall assume that the process $X_t^\gamma$ possesses a Lyapunov function, which is sufficient to ensure the exponential ergodicity of $X_t^\gamma$, \cite{mattingly2002ergodicity,meyn1993stability}, as detailed in the following assumption.

\begin{ass}[Foster--Lyapunov Criterion]
\label{ass:lyapunov}
There exists a function $U:\mathbb{R}^d\rightarrow \mathbb{R}$ and constants $c > 0$ and $b \in \mathbb{R}$ such that $\pi(U) < \infty$ and 
\begin{equation}
\label{eq:lyapunov_condition}
  \mathcal{L} U(x) \leq -c U(x) + b \mathbf{1}_C, \mbox{ and } U(x) \geq 1, \quad x \in \mathbb{R}^d,
\end{equation}
where $\mathbf{1}_C$ is the indicator function over a \emph{petite set}. 
\end{ass}
 For the definition of a petite set we refer the reader to  \cite{meyn1993stability}.
For the generator $\mathcal{L}$ corresponding to the process (\ref{e:nonreversible}), compact sets are always petite. {As noted in \cite{roberts1996exponential}, a sufficient condition on $\pi$  for (\ref{e:nonreversible}) to possess a Lyapunov function is the following.}

\begin{ass}
\label{ass:drift_condition}
{The density $\pi$ is bounded  and} for some $0 < \beta < 1$:
\begin{equation}
\label{eq:lyapunov_condition_nonrev}
\liminf_{|x|\rightarrow \infty} \left((1 - \beta)|\nabla \log \pi(x)|^2 + \Delta \log \pi(x)\right) > 0.
\end{equation}
\end{ass}

\begin{lemma}{{\cite[Theorem 3.3]{roberts1996exponential}}}
Suppose that Assumption \ref{ass:drift_condition} holds then the Foster--Lyapunov criterion  holds for (\ref{eq:sde1}) with 
$$
  U(x) = \pi^{-\beta}(x). $$   Moreover, if the nonreversible term is of the form $\gamma(x) = J\nabla \log\pi(x)$, for $J \in \mathbb{R}^{d\times d}$ antisymmetric, then this choice of $U$ is a Lyapunov function for $X_t^\gamma$ defined by (\ref{e:nonreversible})  for all $\alpha \in \mathbb{R}$.
\end{lemma}
\begin{proof}
For this choice of $\gamma$ the generator of (\ref{e:nonreversible}) has the form
$$
  \cL = (I + \alpha J)\nabla \log\pi(x)\cdot \nabla + \Delta, \quad \alpha \in \mathbb{R},
$$
For $U(x) = \pi^{-\beta}(x)$ we obtain:
\begin{align*}
  \cL U(x) &= -\beta \pi^{-\beta}(x)\nabla \log\pi(x) \cdot(I + \alpha J) \nabla \log \pi(x)  - \beta\nabla\cdot\left(\pi^{-\beta}(x)\nabla\log\pi(x)\right) \\
          &= -\beta\left[\left(1 - \beta\right)\norm{\nabla \log\pi (x)}^2 + \Delta \log\pi(x)\right]\pi^{-\beta}(x).
\end{align*}
Thus, by assumption (\ref{eq:lyapunov_condition_nonrev}), {there exists $\epsilon > 0$ and $M$ such that for $|x| > M$}:
$$
  \left(1 - \beta\right)\norm{\nabla \log\pi (x)}^2 + \Delta \log\pi(x) > \epsilon,
$$
and so
$$
  \cL U(x) \leq -\beta\epsilon U(x) + b\mathbf{1}_{C_M},
$$
where $C_M = \lbrace x \in \mathbb{R}^d \, : \, |x| \leq M \rbrace$ and $b$ is a positive constant.
\\\\
Finally, we note that since $\pi$ is bounded, then $U$ is bounded above from zero uniformly.  Thus, $U$ can be rescaled to satisfy the condition $U \geq 1$, as is required by the Foster--Lyapunov criterion.
\qed
\end{proof}

\begin{remark}
Note that Assumption \ref{ass:drift_condition} holds trivially when $\pi$ is a Gaussian distribution in $\mathbb{R}^d$.  {More generally, following \cite[Example 1]{roberts2002langevin}, consider $\pi(x) \propto \exp(-p(x))$, where $p(x)$ is a polynomial of order $m$ such that $p(x) \rightarrow \infty$ as $|x|\rightarrow \infty$ (necessarily $m \geq 2$ and $m$ is even).  Clearly
$$
 \liminf_{|x|\rightarrow\infty}\frac{|\nabla\log\pi(x)|^2}{|\Delta \log\pi(x)|} = \infty,
$$
and
$$
\liminf_{|x|\rightarrow \infty}(1-\beta)\left|\nabla \log\pi(x)\right|^2 > 0,
$$
so that (\ref{eq:lyapunov_condition}) holds.  On the other hand, consider the case when when $\pi$ is a Student's t-distribution, $\pi(x) \propto (1 + x^2/\nu)^{-\frac{\nu+1}{2}}$ with $\nu \geq 2$. Then it is straightforward to check that (\ref{eq:lyapunov_condition}) will not hold.  Indeed, since $|\nabla \log \pi(x)| \rightarrow 0$, by \cite[Theorem 2.4]{roberts1996exponential} this process is not exponentially ergodic.}
\end{remark}

If condition (\ref{eq:lyapunov_condition}) holds for $X_t^\gamma$, then the process will be exponentially ergodic.  More specifically, the law of the process $X_t^\gamma$ started from a point $x \in \mathbb{R}^d$ will converge exponentially fast in the total variation norm to the equilibrium distribution $\pi$.  In particular, denoting by $(P_t^\gamma)_{t\geq 0}$ the semigroup associated with the diffusion process (\ref{e:nonreversible}), we have the following result.

\begin{theorem}{\cite[Thm 8.3]{cattiaux2012central}, \cite[Thm 3.10, 3.12]{douc2009subgeometric}, \cite[Thm 5.2.c]{down1995exponential}}\\
\label{thm:geometric_ergodicity}
Suppose that Assumption \ref{ass:lyapunov} holds for (\ref{e:nonreversible}), with Lyapunov function $U$. Then there exist positive constants $c$ and $\lambda$ such that, for all $x$,
$$
  \Norm{p^\gamma_t(x, \cdot) - \pi}_{TV} \leq cU(x)e^{-\lambda t},
$$
where $p^\gamma_t(x, \cdot)$ denotes the law of $X_t^\gamma$ given $X_0^\gamma = x$ and $\Norm{\cdot}_{TV}$ denotes the total variation norm.   In particular, for any probability measure $\nu$ such that $U \in L^1(\nu)$, 
$$
  \lim_{t\rightarrow +\infty}\Norm{(P_t^\gamma)^*\nu - \pi}_{TV} = 0,
$$
{where $(P_t^\gamma)^*$ denotes the $L^2(\mathbb{R}^d)$ adjoint of $P_t^\gamma$.}
\end{theorem}
\qed

{For a central limit theorem to hold for the process $X_t^\gamma$, and thus for $\sigma^2_f$ to be finite,  it is necessary that the Poisson equation (\ref{eq:poisson}) is well-posed.  The Foster-Lyapunov condition (\ref{eq:lyapunov_condition}) is sufficient for this to hold.}

\begin{theorem}{\cite[Thm 3.2]{glynn1996liapounov}}
  Suppose that Assumption \ref{ass:lyapunov} holds for the diffusion process (\ref{e:nonreversible}) with Lyapunov function $U$.  Then there exists a positive constant $c$ such that for any $|f|^2 \leq U$,  the Poisson equation (\ref{eq:poisson}) admits a unique zero mean solution $\phi$ satisfying the bound $\norm{\phi(x)}^2 \leq cU(x)$.  In particular, $\phi \in L^2(\pi)$.
\end{theorem}
\qed

The technique of using a Poisson equation to obtain a central limit theorem for an additive functional of a Markov process is widely known, \cite{pardoux2001poisson,bhattacharya1982functional,bhattacharya1985central}.  The approach is based on the fact that, at least formally, we can decompose $\pi_t(f) - \pi(f)$ into a martingale and a ``remainder'' term:
\begin{align*}
  \pi_t(f) - \pi(f) = \frac{1}{t}\int_0^t f(X_s^\gamma)\,ds - \pi(f) &= \frac{\phi(X_0^\gamma) - \phi(X_t^\gamma)}{t} + \frac{\sqrt{2}}{t}\int_0^t \nabla\phi(X_s^\gamma)\cdot dW_s \\
      &=: R_t + M_t.
\end{align*}
Considering the rescaling $\sqrt{t}\left(\pi_t(f) - \pi(f)\right)$, the martingale term $\sqrt{t}M_t$ will converge in distribution to a Gaussian random variable with mean $0$ and variance 
$$
  \sigma^2_f = 2\int_{\mathbb{R}^d} |\nabla \phi(x)|^2\,\pi(dx),
$$
by the central limit theorem for martingales \cite{helland1982central}.  It remains to control the remainder term $\sqrt{t}R_t$.  We distinguish between two cases:  If $X^\gamma_0 \sim \pi$, then since $\phi \in L^2(\pi)$, $\sqrt{t}R_t$ converges to $0$ in $L^2(\pi)$ and the result follows.  In the more general case we must resort to a ``propagation of chaos'' argument (c.f. \cite[Section 8]{cattiaux2012central}), i.e. apply Theorem \ref{thm:geometric_ergodicity} to show that
\begin{equation*}
  \mathbb{E}_{X_0 = x}\left[H\left(\frac{1}{\sqrt{t}}\int_{r}^{t+r}f(X_s^\gamma) - \pi(f)\,ds\right)\right] - \mathbb{E}_{X_0\sim \pi}\left[H\left(\frac{1}{\sqrt{t}}\int_{0}^{t}f(X_s^\gamma) - \pi(f)\,ds\right)\right] \rightarrow 0,
\end{equation*}
as $r \rightarrow \infty$, for all continuous bounded functions $H$.  The result then follows by decomposing
$$
  \pi_{t+r}(f) - \pi(f) = \frac{1}{t}\int_0^r f(X_s^\gamma)\,ds + \frac{1}{t}\int_r^{t+r} f(X_s^\gamma)\,ds,
$$
and applying the propagation of chaos argument to the second term. {The conclusion is summarized in the following result, which provides a central limit theorem for $X^\gamma_t$ starting from an arbitrary initial distribution $\nu$.}

\begin{theorem}{\cite[Thm 4.4]{glynn1996liapounov}}
  If Assumption \ref{ass:lyapunov} holds for Lyapunov function $U$, then for any $f$ such that $f^2(x) \leq U(x)$, there exists a constant $0 < \sigma^2_f < \infty$ such that $\sqrt{t}(\pi_t(f) - \pi(f))$ converges in distribution to an $\mathcal{N}(0, \sigma^2_f)$ distribution, as $t \rightarrow \infty$, for any initial distribution $\nu$,  where
  \begin{equation}
  \label{eq:asympt_var}
    \sigma^2_f = 2\int_{\mathbb{R}^d} \norm{\nabla \phi(x)}^2\,\pi(dx).
  \end{equation}
\end{theorem}
\qed
{In the remainder of this paper we shall study the dependence of $\phi$, and thus $\sigma^2_f$ on the choice of non-reversible perturbation $\gamma$.  We note that (\ref{eq:asympt_var}) is precisely the Dirichlet form associated with the dynamics $\mathcal{L}$ evaluated at the solution $\phi$ of the Poisson equation (\ref{eq:poisson1}).}

\section{Analysis of the Asymptotic Variance}
\label{sec:var_anal}

\subsection{Mathematical Setting}
We present a few definitions from \cite{komorowski2012fluctuations} that will be useful in the sequel.   Let $L^2_0(\pi)$ denote the set of $L^2(\pi)$-integrable functions with zero mean, with corresponding inner product $\langle \cdot, \cdot\rangle_{\pi}$ and norm $\Norm{\cdot}_{\pi}$.  Consider the operator $\cS$ given by (\ref{eq:generator_rev}) densely defined on $L^2_0(\pi)$.  For $k \in \mathbb{N}$, given the family of seminorms
\begin{equation}\label{e:seminorm}
\|f \|_k^2 := \langle f, (-\cS)^k f \rangle_{\pi},
\end{equation}
we define the function spaces 
$$\mathcal{H}^k:= \left\{f \in L^2_0(\pi) \, : \, \|f \|_k < +\infty \right\}.$$  It follows that $\Norm{\cdot}_k$ is a norm on $\cH^k$.  For $k \geq 1$, the norm $\| \cdot \|_k$ satisfies the parallelogram identity and, consequently, the completion of $\cH^k$ with respect to the norm $\| \cdot \|_k$, which is denoted by $\cH^k$, is a Hilbert space. The inner product $\langle
\cdot , \cdot \rangle_k$ in $\cH^k$   is defined through polarization. It is easy to check that, for $f, \, g \in \mathcal{D}(\cL)$,
\begin{equation}\label{e:h1_defn}
\langle f,g \rangle_{1} = \langle f, (-\cS) g \rangle_{\pi} = \langle \nabla f , \nabla g \rangle_{\pi}.
\end{equation}
We note that $\cH^{0} = L_0^2(\pi)$.   A careful analysis of the function space $\cH^k$ is presented in~\cite{komorowski2012fluctuations}. 
\\\\
The operator $\cS$ is symmetric with respect to $\pi$ and can be extended to a selfadjoint operator on $L^2_0(\pi)$, which is also denoted by $\cS$, with domain $\mathcal{D}(\cS) = \cH^2$. We shall make the following assumption on $\pi$ which is required, in addition to Assumption \ref{ass:drift_condition}, to ensure that $\mathcal{L}$ possesses a spectral gap in $L^2_0(\pi)$.
\begin{ass}
\label{ass:logpi}
 \begin{equation}
 \label{eq:logpi}
 \lim_{|x|\rightarrow +\infty}\norm{\nabla\log \pi(x)} = \infty.
 \end{equation}
\end{ass}

In the following lemma we establish a number of fundamental properties relating to the spectrum of $\cL$.  Using these results, we then establish the well-posedness of the Poisson equation (\ref{eq:poisson1}) for any $f \in L^2_0(\pi)$.

\begin{lemma}
\label{lemm:function_space_well_posedness}
Suppose that Assumptions \ref{ass:drift_condition} and \ref{ass:logpi} hold.  Then the embedding $\cH^1 \subset L^2_0(\pi)$ is compact.  For all $\alpha \in \mathbb{R}$, the operator $\mathcal{L}$ satisfies the following Poincar\'{e} inequality in $L^2_0(\pi)$:
\begin{equation}
\label{eq:poincare}
  \lambda \Norm{g}_\pi^2 \leq \langle g, (-\cL)g\rangle_{\pi},   \quad g \in \cH^1,
\end{equation}
{where $\lambda$ is a positive constant, independent of the nonreversible perturbation}.  Moreover,  for all $f \in L^2_{0}(\pi)$ there exists a unique $\phi \in \cH^1$ such that $-\cL \phi = f$.
\end{lemma}
\begin{proof}
  Assumption \ref{ass:drift_condition} clearly implies that 
 \begin{equation}
 \label{eq:logpi2condition}
  -\Delta \log\pi(x) \leq (1-\delta)\norm{\nabla \log \pi(x)}^2 + M_1,  \quad x\in\mathbb{R}^d,
 \end{equation}
 for some $\delta \in (0,1)$ and $M_1 > 0$.  Applying \cite[Theorem 8.5.3]{lorenzi2006analytical}, relations (\ref{eq:logpi}) and (\ref{eq:logpi2condition}) imply the compactness of the embedding $\cH^1 \subset L^2_0(\pi)$.   As a result,  following \cite[Theorem 8.6.1]{lorenzi2006analytical}, this is sufficient for the Poincar\'{e} inequality to hold for $\cS$, i.e.
 $$
    \lambda \Norm{g}_\pi^2 \leq \Norm{g}^2_{1} = \langle g, (-\cS)g \rangle_{\pi},   \quad g \in \cH^1.
 $$
 from which (\ref{eq:poincare}) follows by the antisymmetry of $\cA$ in $L^2_0(\pi)$.   The existence of this Poincar\'{e} inequality implies that the semigroup $P_t^\gamma$ associated with (\ref{e:nonreversible}) converges exponentially fast to equilibrium, that is, for all $f \in L^2_0(\pi)$:
 \begin{equation}
  \label{eq:exp_decay}
    \Norm{P_t^\gamma f}_{\pi} \leq e^{-\lambda t}\Norm{f}_{\pi},  \quad t \geq 0.
 \end{equation}
Given $f \in L^2_0(\pi)$, define 
\begin{equation}
\label{eq:poisson_semigroup}
  \phi(x) = \int_0^\infty P_s^\gamma f(x)\,ds.
\end{equation}
By (\ref{eq:exp_decay}) it follows that $\phi \in L^2_0(\pi)$ {and from (\ref{eq:poisson_semigroup}) we have  $-\cL\phi  = f$}.  Moreover, we have the bound
$$
  \Norm{\nabla\phi}^2_{\pi} = \langle \phi, (-\cL)\phi \rangle_{\pi} = \langle f, \phi \rangle_{\pi}, 
$$
so that
$$
   \Norm{\phi}_{1} \leq C\Norm{f}_{\pi} < \infty.
$$
\qed
\end{proof}

Throughout this section, we shall assume that Assumptions \ref{ass:drift_condition} and \ref{ass:logpi} hold, and moreover, for simplicity we shall make the following additional assumption: 

\begin{ass}
\label{ass:smooth_bounded_gamma_ass}
The nonreversible perturbation $\gamma$ is smooth and bounded in $L^\infty$.
\end{ass}

We believe that this assumption could be relaxed, see in particular \cite{hwang2014variance} for more general assumptions under which the following results should hold.  We stick here to a simple presentation.  In practice, this assumption is not very stringent.  Indeed, suppose that there exists a smooth function $\psi:\mathbb{R}\rightarrow\mathbb{R}_{\geq 0}$ such that
\begin{equation}
\label{eq:condition_bound}
  \psi(V(x))|\nabla V(x)| \leq 1, \quad \mbox{ for all } x \in \mathbb{R}^d.
\end{equation}
Then $\gamma$ satisfies (\ref{e:divergence_free}) since
\begin{align*}
  \nabla\cdot(\pi \gamma) &= -\frac{1}{Z}\nabla\cdot\left(J\nabla e^{-V} \psi(V)\right) \\
  &= -\frac{1}{Z}\nabla\cdot\left(J\nabla e^{-V}\right) \psi(V) - \frac{1}{Z}\,\psi'(V) \nabla V\cdot J\nabla Ve^{-V}\\ &= 0,
\end{align*}
using the fact that $J$ is antisymmetric.  Moreover, it is clear that $\gamma(x) = J \nabla V(x)\psi(V(x))$ is  bounded and smooth, thus satisfying Assumption \ref{ass:smooth_bounded_gamma_ass}.  In particular, if $V$ is a nonnegative polynomial function then condition (\ref{eq:condition_bound}) is satisfied by choosing $\psi$ to be a smooth non-negative function with compact support.  Likewise, it is easy to satisfy (\ref{eq:condition_bound}) if $V$ has compact level sets.
\\\\
We note that this choice of flow $\gamma$ leaves the potential $V$ invariant, i.e. $V(z_t)$ is constant for all $t \geq 0$, where $\dot{z}_t = \gamma(z_t)$.  Thus, for large $|\alpha|$, the flow $\alpha \gamma$ will result in rapid exploration of the level surfaces of $V$, but the motion of $X_t$ between level surfaces is entirely due to the reversible dynamics of the process.   In particular, for potentials with energy barriers, the transition time for $X_t^\gamma$ to cross a barrier will still satisfy the same Arrenhius law as the corresponding reversible process.  Other choices of flow $\gamma$ are possible.  For example, one could alternatively consider a skew-symmetric matrix function $J(x)$ as detailed in \cite{ma2015complete}.  The corresponding flow would then be defined by
\begin{equation}
\label{eq:J_space_dep}
  \gamma(x) = -J(x)\nabla V(x) + \nabla\cdot J(x), \quad x \in \mathbb{R}^d
\end{equation}
It is straightforward to check that $\nabla\cdot\left(\gamma\pi\right) = 0$,
using the fact that $J(x)$ is skew-symmetric.   If additionally, the matrix function $J$ is smooth with bounded derivative and compact support, then $\gamma$ satisfies Assumption \ref{ass:smooth_bounded_gamma_ass}.  As detailed in \cite{ma2015complete} one can further generalise this choice of dynamics by additionally introducing a space dependent diffusion tensor, and an appropriate correction of the drift to maintain ergodicity with respect to $\pi$.  We do not consider this choice of dynamics in this paper, noting that most of the presented results can be readily generalized to this scenario.
\\\\
Given that Assumption \ref{ass:smooth_bounded_gamma_ass} holds, then we have
$$
  \Norm{\cA u}_{\pi} \leq \Norm{\gamma}_{L^\infty(\mathbb{R}^d)}\Norm{u}_{1} \mbox{ and } \int \gamma(x)\cdot \nabla u(x) \pi(dx) = - \int \nabla\cdot\left(\gamma(x)\cdot   \pi(x)\right) u(x)\,dx = 0,
$$
by (\ref{e:divergence_free}),  so that the operator $\cA:\cH^1 \rightarrow L^2_0(\pi)$ defined by $\cA = \gamma\cdot\nabla$ is well defined.  

\subsection{An Expression for the Asymptotic Variance}
The main objective of this paper is to study the effect of the nonreversible perturbation $\gamma$ on the asymptotic variance $\sigma^2_f$, with the aim of choosing $\gamma$ so that $\sigma^2_f$ is minimized.  {Integrating (\ref{eq:asympt_var}) by parts we can express $\sigma^2_f$ in terms of the Dirichlet form for $\mathcal{L}$ as follows:}
\begin{align}
  \notag\frac{1}{2}\sigma_f^2 &= \int \norm{\nabla \phi(x)}^2\, \pi(dx) \\ 
   \notag&= - \int \phi(x)\nabla\cdot\left(\nabla\phi(x) \pi(x)\right)\,dx \\
         \notag&= -\int \phi(x) \left( \Delta \phi(x) + \nabla\log\pi(x)\cdot\nabla \phi(x) \right)\pi(x)\,dx \\
         \notag&= \langle \phi, (-\mathcal{S})\phi \rangle_{\pi} \\
         \label{eq:asympt_var1}&= \langle \phi, (-\mathcal{L})\phi \rangle_{\pi},
\end{align}
where the last line follows from the fact that $\cA$ is antisymmetric in $L^2_0(\pi)$.
\\\\
Starting from (\ref{eq:asympt_var1}) we can obtain a quite explicit characterisation of $\sigma^2_f$.   Indeed, given $f \in L^2_{0}(\pi)$, we can rewrite the last line of (\ref{eq:asympt_var1}) as 
\begin{equation}
\label{eq:sigma_2_f_symm}
  \sigma^2_f = 2\langle f, (-\mathcal{L})^{-1}f \rangle_{\pi} = 2\langle f, \left[(-\mathcal{L})^{-1}\right]^S f\rangle_{\pi},
\end{equation}
where $[\cdot]^{S}$ denotes the symmetric part of the operator.  Using the fact that $-\mathcal{L}$ is invertible on $L^2_0(\pi)$ for all $\alpha \in \mathbb{R}$ by Lemma \ref{lemm:function_space_well_posedness}, the following result yields an expression for $\sigma^2_f(\alpha)$ in terms of $\cS$, $\cA$ and $\alpha$.
\begin{lemma} Let $\cL = \cS + \alpha \cA$, for $\alpha \in \mathbb{R}$.  Then we have
  \begin{equation}
    \label{eq:operator_expansion1}
    \left[(-\cL)^{-1}\right]^{S} = \left[-\cS + \alpha^2\cA^*(-\cS)^{-1}\cA\right]^{-1},
  \end{equation}
  where {$\cA^* = -\cA$} denotes the adjoint of $\cA$ in $L^2_0(\pi)$.  In particular, for all $f \in L^2_0(\pi)$, 
  \begin{equation}
    \label{eq:operator_expansion2}
    \sigma^2_f(\alpha) = 2\langle f,  (-\mathcal{L})^{-1}f \rangle_{\pi} \leq  2\langle f, (-\mathcal{S})^{-1}f \rangle_{\pi} = \sigma^2_f(0).
  \end{equation}
\end{lemma}
\begin{proof}
 Since $-\cS$ is positive, the operator $\mathcal{Q} = (-\cS)^{-\frac{1}{2}}:L^2_0(\pi) \rightarrow \cH^1$ can be defined using functional calculus.  Consider the operator $\mathcal{C} := \mathcal{Q}\cA$.  We can write 
  $-\cS + \alpha^2\cA^*(-\cS)^{-1}\cA =  -\cS + \alpha^2\mathcal{C}^*\mathcal{C}$,
  which can be shown to be closed in $L^2_0(\pi)$ with domain $\cH^2$.  Moreover, this operator has nullspace $\lbrace 0 \rbrace$, so that the inverse is also densely defined on $L^2_0(\pi)$.   To show that (\ref{eq:operator_expansion1}) holds, we expand the left hand side to get
  \begin{align*}
    \left[(-\cL)^{-1}\right]^{S} &= \frac{1}{2}\left[(-\cS + \alpha \cA)^{-1} + (-\cS - \alpha \cA)^{-1}\right]\\
                  &= \frac{1}{2}\cQ\left[ ( I + \alpha \cQ\mathcal{A}\cQ\right)^{-1} + \left( I - \alpha \cQ\mathcal{A}\cQ)^{-1}  \right]\cQ.
  \end{align*}
  Since
  \begin{align*}
    \left( I + \alpha \cQ\mathcal{A}Q\right)^{-1} + \left( I - \alpha \cQ\mathcal{A}\cQ\right)^{-1} = &\left( I + \alpha \cQ\mathcal{A}\cQ\right)^{-1}\left( I - \alpha \cQ\mathcal{A}\cQ\right)^{-1}\left( I - \alpha \cQ\mathcal{A}\cQ\right)\\
      &~ + \left( I - \alpha \cQ\mathcal{A}\cQ\right)^{-1}\left( I + \alpha \cQ\mathcal{A}\cQ\right)^{-1}\left( I + \alpha \cQ\mathcal{A}\cQ\right) \\
      = &\left[I - \alpha^2 \cQ\cA \cQ^2 \cA \cQ\right]^{-1} 2I,
  \end{align*}
  therefore
  \begin{align*}
    \left[(-\cL)^{-1}\right]^{S} &= \cQ\left[I - \alpha^2 \cQ\cA \mathcal{Q}^2 \cA \cQ\right]^{-1}\cQ = \left[-\cS + \alpha^2 \cA^* (-\cS)^{-1} \cA\right]^{-1},
  \end{align*}
  as required.  The inequality (\ref{eq:operator_expansion2}) is then simply a consequence of the fact that 
  $$
    -\cS + \alpha^2 \cA^*(-\cS)^{-1}\cA \succeq -\cS
  $$
  where $\succeq$ denotes the partial ordering between selfadjoint operators on $L^2_0(\pi)$.
  \qed
\end{proof}
$ $
\\
Thus, the asymptotic variance is never increased by introducing a nonreversible perturbation, for all $f \in L^2_0(\pi)$.  {This had already been noted in \cite{rey2014irreversible} where an expression for the the asymptotic was derived as the curvature of the rate function of the empirical measure, and also in \cite{hwang2014variance} using an approach similar to that above.} Expression (\ref{eq:sigma_2_f_symm}) provides us with a formula for $\sigma^2_f$ in terms of a symmetric quadratic form which is explicit in terms of $\cA$ and $\cS$.

\subsection{Quantitative estimates for the Asymptotic Variance
}
In this section we derive quantitative versions of (\ref{eq:operator_expansion1}) and (\ref{eq:operator_expansion2}), using techniques developed in \cite{pavliotis2010asymptotic} for the analysis of the Green-Kubo formula, which is itself based on earlier work on the estimation of the eddy diffusivity in turbulent diffusion \cite{avellanada1,bhattacharya1989asymptotics,majda1993effect,thesis}.    
\\\\
Following the approach of \cite{pavliotis2010asymptotic,golden1983bounds}, to quantify the effect of the antisymmetric perturbation $\alpha \mathcal{A}$ on the asymptotic variance, we define the operator $\cG = (-\cS)^{-1}\cA:\cH^{1}\rightarrow \cH^{1}$.  First we note that we can rewrite (\ref{eq:operator_expansion1}) as follows:
\begin{align}
  \notag\left[(-\cL)^{-1}\right]^{S} &= \left[-\cS + \alpha^2\cA^*(-\cS)^{-1}\cA\right]^{-1} \\
  \label{eq:operator_expansion4}&= \left[I + \alpha^2 (-\cS)^{-1}\cA^*(-\cS)^{-1}\cA\right]^{-1}(-\cS)^{-1} \\
  \notag&= [I - \alpha^2\mathcal{G}^2]^{-1}(-\mathcal{S})^{-1}
\end{align}
and therefore, from (\ref{eq:sigma_2_f_symm}) and (\ref{eq:operator_expansion1}):
\begin{align}
  \notag\frac{1}{2}\sigma^2_f(\alpha) &= \left\langle (-\cS)\hat{f}, \left[I - \alpha^2\cG^2\right]^{-1}(-\cS)^{-1}(-\cS)\hat{f}\right\rangle_{\pi}\\
                 \notag    &=  \left\langle \hat{f},  (-\cS)\left[I - \alpha^2\cG^2\right]^{-1}\hat{f}\right\rangle_{\pi}\\
                 \label{eq:operator_decomposition_G}    &= \left\langle \hat{f}, \left[I - \alpha^2\cG^2\right]^{-1}\hat{f}\right\rangle_{1}.
\end{align}
From the boundedness of the nonreversible perturbation we have the following properties:
\begin{lemma}
  Suppose that Assumption \ref{ass:smooth_bounded_gamma_ass} holds, then the operator $\cG= (-\cS)^{-1}\cA$ is skew-adjoint on $\cH^1$.
\end{lemma}
\begin{proof}
 For $f, g \in \cH^1$:
  $$
    \langle \cG f, g \rangle_{1} =  \langle (-\cS)^{-1}\cA f, (-\cS)g \rangle_{\pi} = -\langle  f, \cA g \rangle_{\pi} =  -\langle  f, (-\cS)(-\cS)^{-1}\cA g \rangle_{\pi} = -\langle  f, \cG  g \rangle_{1},
  $$
  so that $\cG$ is antisymmetric.  Since $\cG$ is also bounded, it follows that $\cG$ is skewadjoint on $\cH^1$. 
  \qed
\end{proof}

Moreover, under appropriate assumptions on the target distribution $\pi$, one can show that the operator $\cG$ is compact. 

\begin{lemma}
\label{lemma:compact}
  Suppose that Assumptions \ref{ass:drift_condition}, \ref{ass:logpi} and \ref{ass:smooth_bounded_gamma_ass} hold, then the operator $\cG$ is compact on $\cH^1$.
\end{lemma} 
\begin{proof}
 By Lemma \ref{lemm:function_space_well_posedness}, the embedding $\cH^1 \subset L^2_0(\pi)$ is compact, and it follows immediately that $\cH^2 \subset \cH^1$ is a compact embedding.  Moreover, since $\gamma$ is bounded we have,
  \begin{align*}
    \Norm{\cG f}_2^2 &= \langle (-\cS)\cG f, (-\cS)\cG f\rangle_{\pi} =  \langle \cA f,\cA f\rangle_{\pi}  \leq \Norm{ f}_{1}^2, \quad  f \in \cH^1,
  \end{align*}
  from which the result follows.\qed
\end{proof}

Extending $\cH^1$ to its complexification, there exists a compact selfadjoint operator $\Gamma$ on $\cH^1$ such that $\cG = i\Gamma$.  {From the spectral theorem for compact selfadjoint operators \cite[Theorem 6.21]{helffer2013spectral}}, the eigenfunctions of $\Gamma$ form a complete orthonormal basis in $\cH^1$, and the eigenvalues of $\Gamma$ are real.  We can partition the eigenfunctions into those spanning $\cN  := {\rm Ker}[\Gamma] = {\rm Ker}[\mathcal{A}]$ and those spanning $\cN^{\bot}$.  We denote by 
$$
  \left\lbrace \lambda_{n}\right\rbrace_{n=1,2,3,\ldots}, \mbox{ and } \left\lbrace e_n \right\rbrace_{n=1,2,3,\ldots},
$$
the  eigenvalues and corresponding eigenfunctions of the operator $\Gamma$ restricted to $\cN^{\bot}$.  For $f \in L^2_0(\pi)$ we have the following unique decomposition for $\hat{f} = (-\cS)^{-1}f$ in $\cH^{1}$:
\begin{equation}
  \hat{f} = \hat{f}_{\mathcal{N}} + \sum_{n=1}^{\infty}\hat{f}_n e_n,
\end{equation}
where $\hat{f}_{\mathcal{N}} \in \mathcal{N}$ and $\hat{f}_n = \langle \hat{f}, e_n \rangle_{1}$.
Consequently, using this spectral decomposition, we can rewrite (\ref{eq:operator_decomposition_G}) as 
\begin{align}
 \notag\sigma^2_{f}(\alpha) &= 2\left\langle \hat{f}, \left[I - \alpha^2\cG^2\right]^{-1}\hat{f}\right\rangle_{1}. \\
                       \notag&= 2\left\langle \hat{f}, \left[I - \alpha^2(i\Gamma)^2\right]^{-1}\hat{f}\right\rangle_{1}\\
                       \notag&= 2\left\langle \hat{f}, \left[I  + \alpha^2\Gamma^2\right]^{-1}\hat{f}\right\rangle_{1}\\
\label{e:deff_expansion}&= 2\Norm{\hat{f}_{\mathcal{N}}}^2_1 + 2\sum_{n=1}^{+\infty} \frac{1}{1 + \alpha^2 \lambda^2_n}\norm{\hat{f}_n}^2.
\end{align}

The conclusion of the above computation is summarized in the following result.

\begin{theorem}\label{thm:deff_expansion}
Suppose that Assumptions \ref{ass:drift_condition}, \ref{ass:logpi} and \ref{ass:smooth_bounded_gamma_ass} hold.  Let $f \in L^2_0(\pi)$ and $\alpha \in \mathbb{R}$, then the asymptotic variance $\sigma^2_f(\alpha)$ corresponding to $X_t^\gamma$ is given by (\ref{e:deff_expansion}).
In particular, we obtain the following limiting values for the asymptotic variance:
\begin{equation}\label{e:small_alpha}
\lim_{\alpha \rightarrow 0} \sigma^2_f(\alpha) = \sigma^2_f(0) = 2\| \hat{f} \|_{1}^2 ,
\end{equation}
and,
\begin{equation}\label{e:large_alpha1}
\lim_{\alpha \rightarrow \pm\infty} \sigma^2_f(\alpha) =  2\| \hat{f}_{\cN}\|^2_{1}. 
\end{equation}
\end{theorem}
\qed
From (\ref{e:deff_expansion}) we obtain the following bounds on the asymptotic variance
\begin{equation*}
\sigma^2_f(\alpha) =  2\| \hat{f}_{\cN}\|^2_{1} + 2\sum_{n=1}^{+\infty} \frac{|\hat{f}_n|^2}{1 + \alpha^2 \lambda_n^2}  \leq  2\| \hat{f}_{\cN}\|^2_{1} + 2\sum_{n=1}^{+\infty} |\hat{f}_n|^2 = 2|| \hat{f} ||_{1}^2 = \sigma^2_{f}(0).
\end{equation*}
\\
The problem of choosing the optimal nonreversible perturbation in (\ref{e:nonreversible}) to minimize the asymptotic variance over all observables in $L^2_0(\pi)$ can be expressed as the following min-max problem: 
$$
  \min_{\gamma \in \mathbb{A}_M}\max_{f \in L^2_0(\pi)}\frac{\left\langle f, (-\cS - \cA(-\cS)^{-1}\cA)^{-1} f \right\rangle_{\pi}}{\Norm{f}_{\pi}^2},\quad \mbox{ where } \cA = \gamma\cdot\nabla\cdot,
$$
where, for some constant $M > 0$,  $\mathbb{A}_M$ denotes the set of admissible nonreversible drifts, typically, 
$$\mathbb{A}_M = \left\lbrace \gamma \in C^\infty_b(\mathbb{R}^d; \mathbb{R}^d) \, : \nabla\cdot(\gamma\pi) = 0 \mbox{ and } \Norm{\gamma}_{L^\infty} < M\right\rbrace.$$
This is equivalent to finding $\gamma \in \mathbb{A}_M$ which solves this max-min problem
\begin{equation}
\label{eq:max_min_variance}
  \max_{\gamma \in \mathbb{A}_M}\min \sigma\left(-\cS + \cA^{*}(-\cS)^{-1}\cA\right),
\end{equation}
where $\sigma[\mathcal{M}]$ denotes the spectrum of the operator $\mathcal{M}$ on $L^2_{0}(\pi)$.  It is important to make the distinction between this problem and that considered in \cite{lelievre2013optimal} for finding the optimal spectral gap, namely finding $\gamma \in \mathbb{A}_M$ such that
\begin{equation}
\label{eq:max_min_spectral}
\max_{\gamma \in \mathbb{A}_M}\min \mbox{Re}\,\left(\sigma\left(-\cS - \cA\right)\right).
\end{equation}
 From  (\ref{eq:max_min_variance}) and (\ref{eq:max_min_spectral}) it is evident that the decrease in asymptotic variance, and the increase in spectral gap arising from a nonreversible perturbation are due to very different mechanisms.  For the latter problem we see that the increase of spectral gap arises from the nonnormality of $\cS + \cA$.  In addition, since the operator $\mathcal{A}^*(-\mathcal{S})^{-1}\mathcal{A}$ is nonnegative in $\cH^1$, a nonreversible perturbation cannot increase the variance.  This is in contrast with the problem of maximising the speed of convergence to equilibrium, as was considered in \cite{lelievre2013optimal}, where increasing the strength of the nonreversible perturbation could result in a decrease of the speed of convergence.
\\\\
We note however that a nonreversible perturbation $\gamma$ which solves the min--max problem (\ref{eq:max_min_variance}) need not be a good candidate for reducing the variance of $\pi_T(f)$ for a given fixed observable $f \in L^2_0(\pi)$.  Indeed, unless $\mathcal{N} = \lbrace 0 \rbrace$ for all $g \in \cH^1$, there will always be an observable for which the nonreversible perturbation does not reduce the variance, since if $g \in \mathcal{N}$ is nonzero, $f = (-\mathcal{S})g$ is nonzero and such that $\sigma^2_f(\alpha) = \sigma^2_f(0)$ for any $\alpha$.  Thus, from a practical point of view, it makes more sense to consider the problem choosing $\gamma$ to minimise the asymptotic variance of $\pi_T(f)$ for a particular observable.   To this end, for a fixed $f \in L^2_0(\pi)$, we can identify two distinct cases:   $(-\cS)^{-1}f \in \mathcal{N}^{\perp}$, in which case  
$$\lim_{|\alpha|\rightarrow\infty}\sigma^2_f(\alpha) = 0,$$ and $(-\cS)^{-1}f \not\in \mathcal{N}^{\perp}$ in which case, $$\lim_{|\alpha|\rightarrow\infty}\sigma^2_{f}(\alpha) =  2\Norm{\hat{f}_{\cN}}^2_{1} > 0.$$

More generally, consider $f \in L^2(\pi)$ so that $f - \pi(f) \in L^2_0(\pi)$.  Assuming we can increase $\alpha$ arbitrarily, {and neglecting any computational issues arising from the resulting discretisation error (which will be discussed in Section \ref{sec:comp_cost})}, the problem of minimising the asymptotic variance of $f$ reduces to finding $\gamma$ such that
\begin{equation}
\label{eq:optimality_criteria}
  (-\cS)^{-1}(f - \pi(f)) \perp \mathcal{N},
\end{equation}
holds in $\cH^1$.  Checking this condition requires the solution of an elliptic boundary value problem, thus this condition is not of practical use.  Nonetheless, we can derive some intuition from (\ref{eq:optimality_criteria}).  Clearly, if we can choose $\gamma$ so that $\mathcal{N} = \lbrace 0 \rbrace$ in $\cH^1$, the nonreversible perturbation will be optimal for all observables $f \in L^2(\pi)$.  In general, it might not be possible to find $\gamma$ that satisfies this condition.  For the nonreversible drift considered in \cite{lelievre2013optimal}, namely $\gamma = J\nabla V$, with $J^\top = -J$ , $\mathcal{N}$ will always be nontrivial; indeed, in this case  $H\circ V \in \mathcal{N}$ for all functions $H$ such that $H\circ V \in \cH^1$.   In this case, it is always possible to choose an observable $f$ such that $\gamma$ will not be optimal for $f$, in the sense that $\sigma_f^2(\alpha)$  is nonzero in the limit $\alpha\rightarrow\infty$.  We also remark that these asymptotic results, in particular the distinction between these two cases is reminiscent of similar results that have been obtained in the context of turbulent diffusion \cite{majda1993effect,thesis}. 
\\\\
An analogous classification of the asymptotic behaviour of the spectral gap of the operator (\ref{eq:operator_decomposition}) in the limit of large $\alpha$  is considered in \cite{franke2010behavior} on a compact manifold.  Indeed, in \cite[Theorem 1]{franke2010behavior} it is determined that the spectral gap is finite in the limit of $\alpha\rightarrow \pm\infty$ if and only if $\cA$ has a nonconstant eigenfunction in $\cH^1$.  However, one should note that, the asymptotic variance $\sigma^2_f(\alpha)$ may converge to $0$ as $\alpha\rightarrow \pm\infty$ even when the spectral gap is finite, see also \cite[Example 2.9]{rey2014irreversible} for a counterexample.

\subsection{A Two Dimensional Example}
In this section we present a simple example on $\mathbb{R}^2$.  Consider the problem of calculating $\pi(f)$ with respect to the Gaussian distribution $\pi(x) = \frac{1}{Z}e^{-|x|^2/2}$.  The reversible overdamped Langevin equation corresponding to $\pi$ is given by the OU process:
\begin{equation}
\label{eq:gaussian_2d_sde}
dX_t = -X_t\,dt + \sqrt{2}\,dW_t.
\end{equation}
We introduce an antisymmetric perturbation $\alpha \gamma(X_t)$ where $\gamma(x)$ is the flow given by
$$
  \gamma(x) = \left(\begin{array}{cc} x_2 \\ -x_1 \end{array}\right). \\ 
$$
The infinitesimal generator for the perturbed nonreversible dynamics is
$$
  \mathcal{L} = \mathcal{S} + \alpha \mathcal{A},
$$
where 
$$
  \mathcal{S}f(x) = -x\cdot\nabla f(x) + \Delta f(x) \quad \mbox{ and } \mathcal{A}f(x) = \gamma(x)\cdot\nabla f(x).
$$
In polar coordinates, this is given by
$$
  \mathcal{L}f(r,\theta) = \left( -r + \frac{1}{r}\right)\partial_r f(r, \theta) + \partial_{r,r}f(r, \theta) + \alpha\partial_{\theta}f(r, \theta) + \frac{1}{r^2}f_{\theta,\theta}(r, \theta),\quad (r,\theta) \in \mathbb{R}_{> 0}\times (0,2\pi).
$$
As $|\alpha| \rightarrow +\infty$, we expect the deterministic flow to move increasingly along the level curves of the distribution $\pi(x)$.  Given an observable $f$, any variance in $\pi_T(f)$ arising from the variation of $f$ along the level curves should vanish as $|\alpha|\rightarrow \infty$,  leaving only the variance contributed by the variation of $f$ between level curves.  We make this precise with a particular example. Consider the observable $f(x_1,x_2) = 2 x_1^2$, expressible in polar coordinates as
$$
  f(r, \theta) = 2r^2\cos^2(\theta) = r^2\left( 1+ \cos(2 \theta)\right).
$$ 
Noting that $\pi(f) = 2$, it is straightforward to check that the Poisson equation $-\mathcal{L}\phi = f - \pi(f)$, has mean zero solution given by
$$
  \phi(r, \theta) = \left(\frac{r^2}{2} - 1\right)  + \frac{r^2}{2}\frac{\left(\cos(2\theta) - \alpha \sin(2\theta)  \right)}{1 + \alpha^2}
$$
The asymptotic variance can be evaluated directly as follows
\begin{align*}
  \sigma^2_f &= 2\left\langle \phi, (-\mathcal{L})\phi\right\rangle_{\pi} \\
             &= \frac{2}{Z}\int_0^{2\pi}\int_{0}^{\infty} \left(\left(\frac{r^2}{2} - 1\right)  + \frac{r^2}{2}\frac{\left(\cos(2\theta) - \alpha \sin(2\theta)  \right)}{1 + \alpha^2}\right)\left(r^2(1 + \cos(2\theta)) - 2\right)r e^{-r^2/2}\,dr\,d\theta \\
             &= 4\left( 1 + \frac{1}{1 + \alpha^2}\right).
\end{align*}
Therefore, as $|\alpha|\rightarrow \infty$, the asymptotic variance converges to $4$.  We note that, in this case,
$$
  \hat{f} = (-\mathcal{S})^{-1} \left(f - \pi(f)\right) =  \left(\frac{r^2}{2} - 1\right)  + \frac{r^2}{2}\left(\cos(2\theta) \right),
$$
where  $\left({r^2}/{2} - 1\right)$  is perpendicular to ${r^2}/{2} \cos(2\theta)$ in $\cH^1$.  Since the nullspace of $\mathcal{A}$ consists of all $\cH^1$ functions which depend only on $r$, we have $\hat{f}_{\mathcal{N}} = r^2/2 - 1$.  It follows that $2\Norm{\hat{f}_{\mathcal{N}}}^2_{1} = 4$, so that
$$
  \lim_{|\alpha|\rightarrow \infty}\sigma^2_f(\alpha) = 2\Norm{\hat{f}_{\mathcal{N}}}^2_{1},
$$
which agrees with the conclusions of Theorem \ref{thm:deff_expansion}.  As $|\alpha|\rightarrow \infty$, the motion in the $\theta$ direction is averaged out.  Indeed, in this limit, $2\Norm{\hat{f}_{\mathcal{N}}}^2_{1}$ corresponds to the asymptotic variance of the observable $\int_0^T r_t^2 \,dt$, where $r_t$ is the following 1D reversible process,
$$
dr_t = \left(-r_t + \frac{1}{r_t}\right)\,dt + \sqrt{2}\,dW_t,  \quad r_0 > 0.
$$
For more general observables $f(r, \theta)$, we expect maximum variance reduction as $|\alpha|\rightarrow \infty$ when $f$ varies strongly with respect to $\theta$.  In the other extreme, we expect zero improvement when $f$ depends only on~$r$.  This intuition is formalised in the following section, in particular in Proposition \ref{prop:gaussian_case} and the subsequent bound (\ref{eq:var_bound2}).
\\\\
For more general potentials, using a flow field of the form $\gamma(x) = J\nabla \log \pi(x)$, $J^\top = -J$, the mechanism for reducing the asymptotic variance is analogous: the large antisymmetric drift gives rise to fast deterministic mixing along the level curves of the potential,  while the reversible dynamics induce slow diffusive motion along the gradient of the potential.  When the nullspace of $\mathcal{A}$ is trivial in $\cH^1$, the fast deterministic flow is ergodic, so that, for $\alpha$ large, the antisymmetric component will cause a rapid exploration of the entire state space. Consequently, the asymptotic variance converges to $0$ as $\alpha \rightarrow \infty$.  On the other hand, if $\mathcal{A}$ has a nontrivial nullspace, the antisymmetric perturbation is no longer  ergodic, and the state space can be decomposed into components such that the rapid flow behaves ergodically in each individual component.  In the limit of large $\alpha$, $X_t$ becomes a fast-slow system, with rapid exploration within the ergodic components coupled to a slow diffusion between components.  Very recently, Rey-Bellet and Spiliopoulos \cite{rey2014variance} have applied Freidlin-Wenzell theory to rigorously analyse this case in the large $\alpha$ limit for a large class of potentials. 

\section{Nonreversible Perturbations of Gaussian Diffusions}\label{sec:gauss cxx}
For the case when the target distribution is Gaussian, the SDE (\ref{e:nonreversible}) for $X^\gamma_t$ is linear.  In this case,  we can obtain an explicit analytical expression for the asymptotic variance for a large class of observables $f$.  Indeed, consider the nonsymmetric Ornstein-Uhlenbeck process in $  \mathbb{R}^d$:
\begin{equation}
\label{eq:sde_linear}
  dX^\gamma_t = -(I + \alpha J)X^\gamma_t\,dt + \sqrt{2}\,dW_t,
\end{equation}
where $J$ is an antisymmetric matrix, $\alpha > 0$, and $W_t$ is a standard $d$-dimensional Brownian motion.  The stationary distribution $\pi(x)$ is $\mathcal{N}(0, I)$, independent of $\alpha$ and $J$. Although this system does not fall under the framework of Theorem \ref{thm:deff_expansion} we are still able to obtain analogous conditions for a reduction in the asymptotic variance. The objective of this section is to mirror the results for speeding up convergence to equilibrium of $X_t$ that were derived in \cite{lelievre2013optimal} to the case of minimizing the asymptotic variance.  In particular, following arguments similar to \cite[Section 4.2]{pavliotis2010asymptotic}, an explicit formula for the asymptotic variance will be derived, from which an optimal $J$ can be chosen, in a manner similar to \cite{lelievre2013optimal}.  We note that for the process (\ref{eq:sde_linear}), the optimal nonreversible perturbation obtained in \cite{lelievre2013optimal} does not provide any increase to the rate of convergence to equilibrium, since all eigenvalues of the covariance matrix of the Gaussian stationary distribution are the same.  Nonetheless, in this section we show that for certain observables, the asymptotic variance of $\pi_T(f)$ can be dramatically decreased.

\subsection{Explicit formula for the asymptotic variance}
We shall assume that the observable $f$ is a quadratic functional of the form 
\begin{equation}
\label{eq:quadratic_observable}
  f(x) = x\cdot M x + l\cdot x + k,
\end{equation}
where $M \in \mathbb{R}^{d\times d}$ is a symmetric positive definite matrix, $l \in \mathbb{R}^d$ and $k$ is a constant, chosen so that $f(x)$ is centered with respect to $\pi(x)$: $k = -\Tr M$.  Consider the Poisson equation (\ref{eq:poisson}):
\begin{equation}
\label{eq:poisson_quadratic}
-\mathcal{L}\phi(x) =x\cdot M x + l\cdot x - \Tr M, \quad \pi(\phi) = 0,
\end{equation}
where $\mathcal{L}$ is the infinitesimal generator of (\ref{eq:sde_linear}) given by $\mathcal{L} = -(I + \alpha J)x\cdot\nabla + \Delta$.  For such observables we can solve the Poisson equation (\ref{eq:poisson_quadratic}) analytically and obtain a closed-form formula for the asymptotic variance for the observable $f$.

\begin{proposition}
\label{prop:gaussian_case}
   Let $A = (I + \alpha J)^\top = (I - \alpha J)$.  The  unique mean zero solution of the Poisson equation (\ref{eq:poisson_quadratic}) is given by
  $$
    \phi(x) = x\cdot C x + D\cdot x - \Tr(C),
  $$
  where 
  \begin{equation}
    C = \frac{1}{2}\int_0^\infty e^{-A s} M e^{-A^\top s}\,ds,
  \end{equation}
  and 
  $$
    D = A^{-1}l.
  $$
  Moreover, the asymptotic variance is given by
  $$
    \sigma^2_{f}(\alpha) = \Norm{M}^{2}_{F} - \int_0^\infty e^{-2s}\Norm{[M, e^{-\alpha J s} ]}_F^2 \, ds + 2\,l\cdot(I + \alpha^2 J^\top J)^{-1}l,
  $$
where $\Norm{A}_F = \sqrt{\mbox{Tr}[AA^\top]}$ denotes the Frobenius norm of $A$, and $[A, B] = AB - BA$ is the commutator of $A$ and $B$.  In particular,
\begin{equation}
\label{eq:variance_bound_quadratic}
  \sigma^2_{f}(\alpha) \leq \sigma^2_{f}(0), \quad \forall \alpha\in\mathbb{R}.
\end{equation}
\end{proposition}
\begin{proof}
Clearly $\phi(x)$ must also be a quadratic function of $x$, so we make the ansatz 
$$
  \phi(x) = x\cdot C x + D\cdot x - \Tr C.
$$
Plugging this into (\ref{eq:poisson_quadratic}) we obtain
$$
  x \cdot A\left((C + C^\top)x + D\right) - \Tr (C + C^\top) =  x\cdot M x + l\cdot x -\Tr M.
$$
Comparing equal powers of $x$ we have
\begin{subequations}
\begin{equation}
  \label{eq:o2_poisson}
  x\cdot A(C + C^\top)x = x\cdot M x,
\end{equation}
\begin{equation}
  \label{eq:o1_poisson}
  AD\cdot x = l\cdot x
\end{equation}
\begin{equation}
  \label{eq:o0_poisson}
  \Tr C = \frac{1}{2}\Tr M
\end{equation}
\end{subequations}
for all $x \in \mathbb{R}^d$.  Since $x$ is arbitrary, it follows that
$$
  D = A^{-1}l.
$$
Equation (\ref{eq:o2_poisson}) is a Lyapunov equation, which is well-posed as $M$ is positive definite and $\mbox{spec}(A) \subset \lbrace \lambda \in \mathbb{C} \, : \mbox{Re} \lambda >0 \rbrace $.  Indeed, for $C$ given by
$$
  C = \frac{1}{2}\int_0^\infty e^{-A s}M e^{-A^\top s}\,ds,
$$
both  (\ref{eq:o2_poisson}) and (\ref{eq:o0_poisson}) are satisfied.  The asymptotic variance $\sigma^2_f(\alpha)$ is then given by (see (\ref{eq:asympt_var1})),
\begin{equation}
\begin{aligned}
\frac{1}{2}\sigma^2_{f}(\alpha) = \langle \phi(x), f(x) \rangle_{\pi} = &\left\langle {x\cdot C x} + D\cdot x - \Tr C, {x\cdot M x} + l\cdot x - {\Tr M}\right\rangle_{\pi} \\
= &\langle {x\cdot C x}, x \cdot M x\rangle_\pi + \langle D\cdot x, l\cdot x \rangle_\pi\\
 & - \langle \Tr C, x\cdot M x\rangle_\pi - \langle x\cdot C x, \,\Tr M \rangle_\pi + \Tr C \Tr M.
\end{aligned}
\end{equation}
Using the fact that $C$ and $M$ are symmetric,
$$
\langle {x\cdot C x}, x \cdot M x\rangle_\pi  = \sum_{i,j,k,l} C_{ij}M_{kl}\mathbb{E}_\pi \left[x_i x_j x_k x_l\right] = \,\Tr C\,\Tr M + 2\,\Tr(CM^\top).
$$
Therefore,
\begin{equation}
\begin{aligned}
\frac{1}{2}\sigma^2_f(\alpha) = &2\,\Tr(CM^\top) + \langle D\cdot x, l\cdot x\rangle_\pi \\
        = &2\,\Tr(CM^\top) + D\cdot l\\
        = \label{eq:asympt_var_quadratic}& \,  \int_0^\infty \Tr\left[e^{-As}Me^{-A^\top s}M^\top\right]\,ds + l\cdot A^{-1} l.
\end{aligned}
\end{equation}
Since $A = I - \alpha J$, where $J$ is antisymmetric, we have that, for all $l \in \mathbb{R}^d$:
\begin{align*}
  l\cdot A^{-1} l = &\frac{1}{2}l\cdot\left[(I - \alpha J)^{-1} + (I + \alpha J)^{-1}\right]l \\
  = &\frac{1}{2}l\cdot\left[(I - \alpha J)^{-1}(I + \alpha J)^{-1}(I + \alpha J) + (I + \alpha J)^{-1}(I - \alpha J)^{-1}(I - \alpha J)\right]l \\
  = &l\cdot(I + \alpha^2 J^\top J)^{-1}l.
\end{align*}
Therefore, we obtain the following expression for $\sigma^2_{f}(\alpha)$: 
\begin{equation}
\label{eq:gaussian_asympt_var}
\sigma^2_{f}(\alpha) = 2\int_0^\infty e^{-2s}\Tr\left[e^{\alpha Js} M e^{-\alpha J s}M^\top\right]\,ds + 2l\cdot(I + \alpha^2 J^\top J)^{-1} l.
\end{equation}
Writing the first term as follows:
\begin{align*}
\int_0^\infty e^{-2s}\Tr[e^{\alpha J s}M e^{-\alpha J s}M^\top]\,ds =  \frac{1}{2}\lVert M \rVert^2_F - \frac{1}{2}\int_0^\infty e^{-2s}\left\lVert \left[M, e^{-\alpha J s} \right]\right\rVert^2_F \,ds.
\end{align*}
From this, and the facts that $J^\top J \geq 0$ and $[M, I] = 0$, the inequality (\ref{eq:variance_bound_quadratic}) follows.
\qed
\end{proof}

Since $ J$ is skew-symmetric,  the matrix exponential $e^{\alpha J s}$ is a rotation matrix.  Thus, the matrix 
$e^{\alpha J s} M e^{-\alpha J s}$ has the same eigenvalues as $M$ and $M^\top$.  From \cite[ III.6.14]{bhatia1997matrix} we have
$$
  \lambda^{\downarrow}(M)\cdot\lambda^{\uparrow}(M) =  \lambda^{\downarrow}(e^{\alpha J s}M\,e^{-\alpha J s})\cdot\lambda^{\uparrow}(M^\top) \leq \Tr\left[e^{\alpha J s} M e^{-\alpha J s}M^\top\right],
$$
where $\lambda^{\uparrow}(M)$ and $\lambda^{\downarrow}(M)$ denote the vectors of eigenvalues of $M$, sorted in ascending and descending order, respectively.  In particular,
\begin{equation}
\label{eq:var_bound1}
   \int_0^\infty e^{-2s}\Tr\left[e^{\alpha Js} M e^{-\alpha J s}M^\top\right]\,ds  \geq \frac{1}{2}\lambda^{\downarrow}(M)\cdot\lambda^{\uparrow}(M).
\end{equation}
On the other hand, let $\mathcal{N} = \mathcal{N}[J] = \mathcal{N}[J^\top J]$, and diagonalising $J^\top J$ we obtain
$$
J^\top J = U^\top D U, \quad \mbox{ where } D = \left[\begin{array}{c|c}\mathbf{0} & \mathbf{0} \\ \hline \mathbf{0} & \alpha\mathbf{\Lambda}\end{array}\right], \quad \mbox{ for } \mathbf{\Lambda} = \mbox{diag}\left(\lambda_1, \ldots, \lambda_{d-N}\right).
$$
where $U$ is a $d\times d$ orthogonal matrix with the first $N$ columns spanning $\mathcal{N}$.  Then
$$
   l\cdot(I + \alpha^2 J^\top J)^{-1} l= Ul\cdot\left[\begin{array}{c|c}\mathbf{I} & \mathbf{0} \\ \hline \mathbf{0} & (\mathbf{I} + \alpha\mathbf{\Lambda})^{-1}\end{array}\right] Ul \xrightarrow{\alpha\rightarrow\pm \infty} Ul\cdot\left[\begin{array}{c|c}\mathbf{I} & \mathbf{0} \\ \hline \mathbf{0} & \mathbf{0}\end{array}\right] Ul = \norm{l_{\mathcal{N}}}^2,
$$
where $l_{\cN}$ is the projection of $l$ onto $\cN$.  Thus, for quadratic observables, from (\ref{eq:gaussian_asympt_var}), the asymptotic variance  has the following lower bound in the limit of large $\alpha$:
\begin{equation}
\label{eq:var_bound2}
  \lim_{\alpha\rightarrow \infty}\sigma^2_f(\alpha)  \geq \underline{\sigma}^2_{f} := \lambda^{\downarrow}(M)\cdot\lambda^{\uparrow}(M) + 2\norm{l_{\cN}}^2,
\end{equation}
In particular, for diffusions with linear drift, the asymptotic variance cannot be decreased arbitrarily.  The worst case scenario is when $f(x) = m\left(\norm{x}^2 - d\right)$, for some $m > 0$, for which the asymptotic variance will not be decreased, for any antisymmetric matrix $J$ and $\alpha \in \mathbb{R}$, analogous to the situation which occurs in \cite{lelievre2013optimal} for maximising the spectral gap, when all the eigenvalues of the covariance matrix are equal.
\\\\
\subsection{Finding the Optimal Perturbation}
We first focus on the case when the observable is linear, i.e. $f(x) = l\cdot x$, similar to that considered in \cite[Section 4.2]{pavliotis2010asymptotic},  so that $\sigma^2_f = 2l\cdot(I + \alpha^2 J^\top J)^{-1} l$.    Suppose that we fix $\alpha \in \mathbb{R}$, and wish to choose $J$ which minimizes the asymptotic variance subject to $\Norm{J}_F = 1$.  In this case, we have the equality $$\lim_{\alpha\rightarrow\pm\infty}\sigma^2_f = 2\norm{l_{\mathcal{N}}}^2,$$ which is optimal if $l$ is orthogonal to $\mathcal{N}$.  Thus, the best we can do is to choose $J$ such that $l$ is an eigenvector of $J^\top J$ with maximal eigenvalue.  This can be done by choosing a unit vector $\omega \in \mathbb{R}^d$ orthogonal to $l$ and setting
$$
  J = \frac{\left(\tilde{l}\otimes \omega -  \omega \otimes \tilde{l}\right)}{\sqrt{2}},
$$
where $\tilde{l} = l/|l|$
Then,
$$
  J^* J = \frac{1}{2}\left(-\tilde{l} \otimes \omega +  \omega \otimes \tilde{l}\right)\left(\tilde{l} \otimes \omega -  \omega \otimes \tilde{l}\right) = \frac{\tilde{l}\otimes \tilde{l} + \omega \otimes \omega}{2},
$$
is a projector onto $\lbrace {l}, \omega\rbrace$ and $\Norm{J}_{F} = 1$, where $\Norm{\cdot}_{F}$ is the Frobenius norm with respect to the Euclidean basis. In this case 
$$
  \sigma^2_f(\alpha) = 2l\cdot(I + \alpha^2 J^\top J)^{-1}l = \frac{4\norm{l}^2}{2 + \alpha^2}.
$$
We note in addition that, when minimising the asymptotic variance, there are many antisymmetric matrices $J$ which give the minimal asymptotic variance.  As an example, let $d = 3$, consider the observables $f(x) = l^{(i)}\cdot x$ for  $l^{(1)} = (0, 1, 1)^\top/\sqrt{2}$ and $l^{(2)} = (1, 0, 1)^{\top}/\sqrt{2}$ and $l^{(3)} = (1, -1, 1)^{\top}/\sqrt{3}$, respectively.  We choose $J$ to be 
\begin{equation}
\label{eq:J_linear}
  J = \frac{1}{\sqrt{6}}\left(\begin{matrix} 0 & 1 & 1 \\ -1 & 0 & 1 \\ -1 & -1 & 0\end{matrix}\right).
\end{equation}
In this case 
$$
\sigma^2_{f}(\alpha) = \frac{2}{6 + 3\alpha^2}l\cdot\left(\begin{matrix}6 + \alpha^2 & -\alpha^2 & \alpha^2 \\ -\alpha^2 & 6 + \alpha^2 & -\alpha^2 \\ \alpha^2 & -\alpha^2 & 6 + \alpha^2\end{matrix}\right)l.
$$
The decay of $\sigma^2_{f}(\alpha)$ as $\alpha\rightarrow \infty$ is strongly dependent on the nullspace of $J$, given by
$$ \mathcal{N} = \mbox{span}[\zeta], \quad \mbox{ where } \zeta = (1, -1, 1).$$  In Figure \ref{fig:linear_variance} we plot the asymptotic variance for these observables.   For $f(x) = l^{(1)}\cdot x$, see that as $\alpha \rightarrow \infty$ the asymptotic variance converges to $0$. This is due to the fact that $l^{(1)}$ is perpendicular to the nullspace of $J$. Thus, for this observable the matrix $J$ given by (\ref{eq:J_linear}) is an optimal perturbation.  For $f(x) = l^{(2)}\cdot x$, since $l^{(2)}$ is not orthogonal to $\zeta$, as $\alpha \rightarrow \infty$,
$$
  \sigma^2_{f}(\alpha) \rightarrow 2\norm{l_{\cN}^{(2)}}^2 =   \frac{4}{3}.
$$
Finally, for $f(x) = l^{(3)}\cdot x$ we observe that the asymptotic variance remains constant at $2$, since $l^{(3)} \in \mathcal{N}$, the nonreversible perturbation has no effect on this observable.
\\\\
\begin{figure}[ht]
\begin{subfigure}[b]{0.5\textwidth}
                \includegraphics[width=\textwidth]{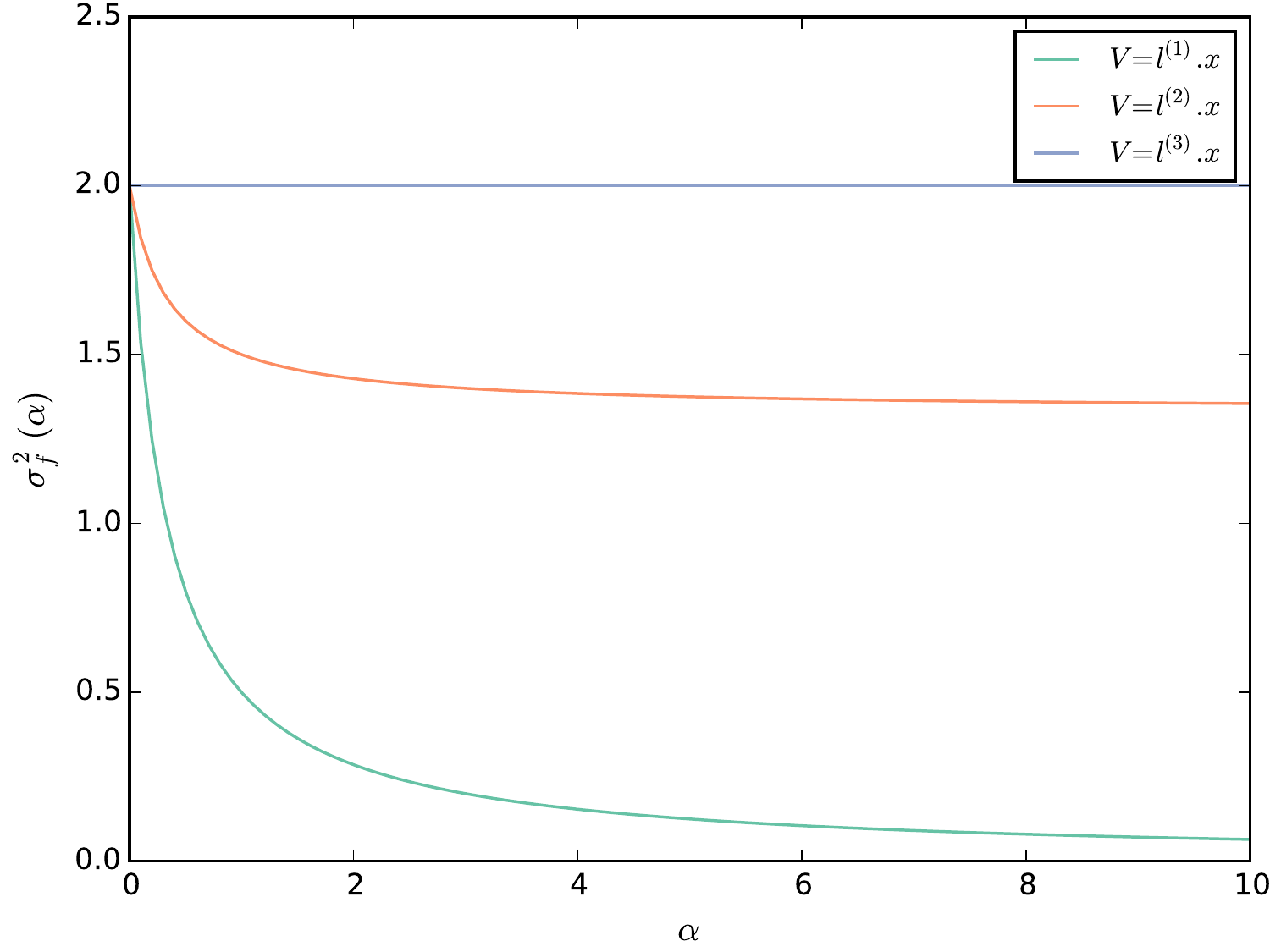}
                \caption{\label{fig:linear_variance}Linear observable}
        \end{subfigure}%
      ~
      \begin{subfigure}[b]{0.5\textwidth}
                \includegraphics[width=\textwidth]{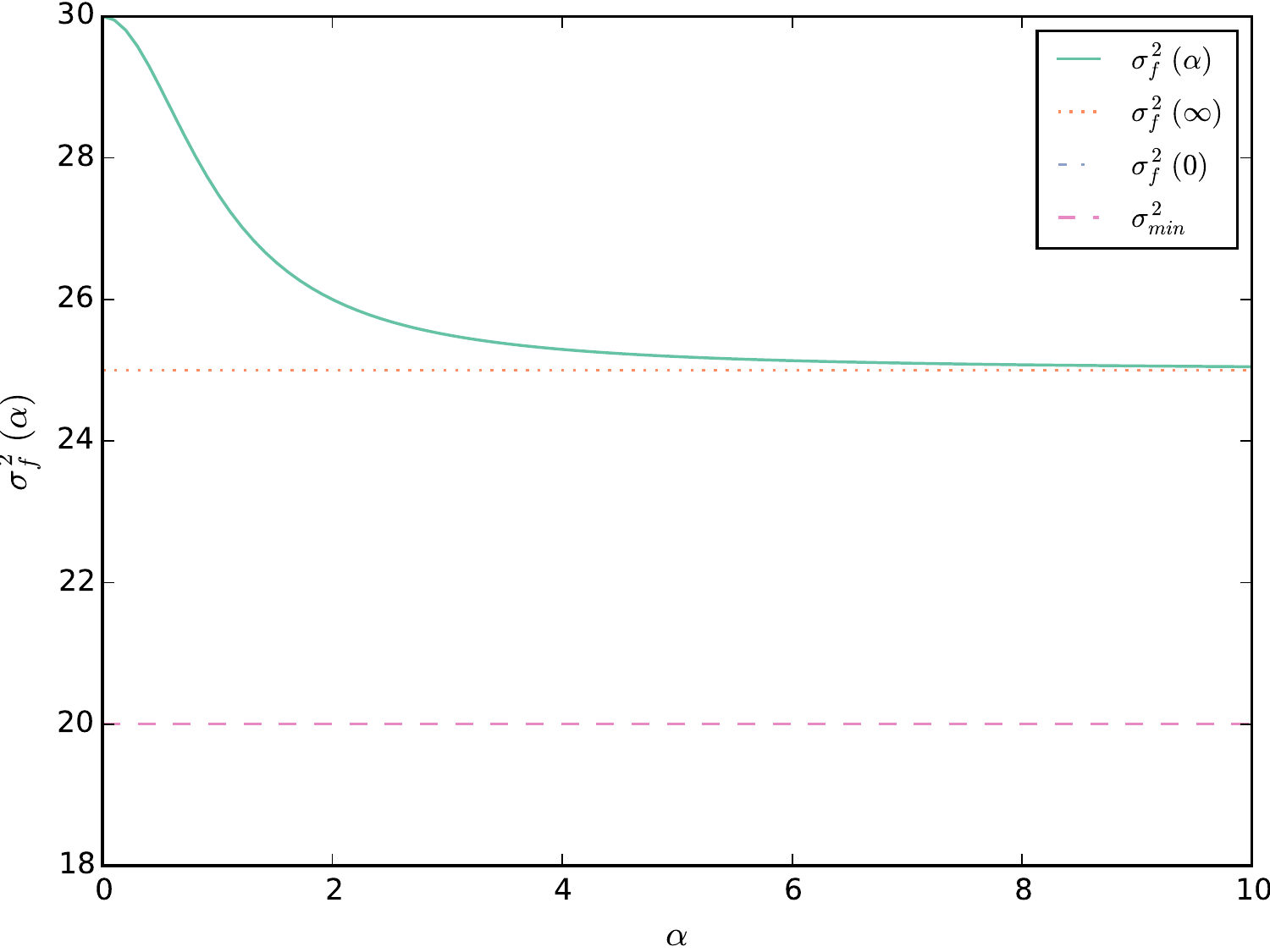}
                \caption{\label{fig:quadratic_variance}Quadratic observable}
        \end{subfigure}%
\caption{The asymptotic variance for a linear diffusion $X_t$ given in (\ref{eq:sde_linear}) for linear observables $f(x) = l^{(i)}\cdot x$, $i= 1,2,3$ (left), and for a quadratic observable $f(x) = x\cdot M_1 x$ (right).}
\label{fig:optimal_variances}
\end{figure}

We now focus on the case when $l = 0$ so that $f(x) = x\cdot M x - \Tr(M)$.  In this case, it is not clear how to construct an optimal $J$, however the bound (\ref{eq:var_bound2}) suggests a good candidate for $J$.  Suppose that $M$ has eigenvalues $\lambda_1 \leq \lambda_2 \leq \ldots \leq\lambda_d$ with corresponding eigenvectors $e_1, \ldots, e_d$.  Suppose that $d$ is even.  Let 
$$\lbrace i_1, j_1 \rbrace, \ldots ,\lbrace i_{ d/2 }, j_{ d/2 } \rbrace,$$
be a partition of $1, \ldots d$ into disjoint pairs.  Define the antisymmetric matrix $J$ by
\begin{equation}
\label{eq:J_quasi_opt}
J = \sum_{k=1}^{d/2}e_{i_k}\otimes e_{j_k} - e_{j_k}\otimes e_{i_k}, \quad k =1, \ldots, \frac{d}{2}.
\end{equation}
In this case $e^{\alpha Js}$ can be decomposed into a product of rotations between the pairs of eigenvectors:
$$
e^{\alpha J s} = \prod_{k=1}^{d/2} R_{e_{i_k}, e_{j_k}}(\alpha s),
$$
where $R_{v, w}(\theta)$ is an anticlockwise rotation of angle $\theta$ in the $\lbrace v, w \rbrace$ plane.  Then
\begin{align*}
\int_{0}^\infty e^{-2s}\Tr[e^{\alpha J s}M e^{-\alpha J s}M^\top]\,ds = &\frac{1}{4(1+\alpha^2)}\sum_{k=1}^{d/2}(\lambda_{i_k} - \lambda_{j_k})^2 + \frac{1}{4}\sum_{k=1}^{d/2}(\lambda_{i_k} + \lambda_{j_k})^2.
\end{align*}
For this choice of $J$, the asymptotic variance is given by
$$
\lim_{\alpha \rightarrow \pm\infty}\sigma^2_{f}(\alpha) = \frac{1}{2}\sum_{k=1}^{d/2}(\lambda_{i_k} + \lambda_{j_k})^2.
$$
Arguing by contradiction, this is minimized by choosing $i_k = k$,  and $j_k = (d - k + 1)$ for $k = 1,\ldots  ,d/2 $, in which case
\begin{align*}
\lim_{\alpha \rightarrow \pm\infty}\sigma^2_{f}(\alpha) = &\frac{1}{2}\left[(\lambda_1 + \lambda_d)^2 + (\lambda_2 + \lambda_{d-1})^2 + \ldots + (\lambda_{ d/2 } + \lambda_{d/2 + 1})^2  \right] \\
 = &\frac{1}{2}\lambda^{\downarrow}(M)\cdot\lambda^{\uparrow}(M) + \frac{1}{2}\sum_{k=1}^d \lambda_k^2.
\end{align*}
Clearly,
$$
  \lambda^{\downarrow}(M)\cdot\lambda^{\uparrow}(M) \leq \frac{1}{2}\lambda^{\downarrow}(M)\cdot\lambda^{\uparrow}(M) + \frac{1}{2}\sum_{k=1}^d \lambda_k^2 \leq \Tr(M^2),
$$
however these bounds are only tight  when  $M = m\, I$, for $m > 0$.
\\\\ 
As an example, on Figure \ref{fig:quadratic_variance} we plot the asymptotic variance for a quadratic observable $x\cdot M_1 x$ for a linear diffusion (\ref{eq:sde_linear}), where $d=4$.  The matrix $M_1$ is given by
$$
  M_1 = \left(\begin{array}{cccc} 3/2 & -1/2 & 0 & 0 \\ -1/2 & 3/2 & 0 & 0 \\ 0 & 0 & 7/2 & -1/2 \\ 0 & 0 & -1/2 & 7/2\end{array}\right),
$$
with eigenvalues $\lambda_i = i$, for $i=1,\ldots 4$.  When $\alpha = 0$, the asymptotic variance is $\sigma^2_f(0) = 30$.  The lower bound (\ref{eq:variance_bound_quadratic}) is given by $\underline{\sigma}^2_f = 2\left(\lambda_1\lambda_4 + \lambda_2\lambda_3\right) = 20$.  We choose the antisymmetric matrix $J:\mathbb{R}^{4\times 4}$ as in (\ref{eq:J_quasi_opt}), that is
$$
  J = \frac{1}{2}\left(\begin{array}{cccc} 0 & 0 & 1 & 0 \\ 0 & 0 & 0 & -1 \\ -1 & 0 & 0 & 0 \\ 0 & 1 & 0 & 0\end{array}\right).
$$
The resulting asymptotic variance $\sigma^2_{f}(\alpha)$ is plotted as a function of $\alpha$ in Figure \ref{fig:linear_variance} which converges to $\frac{1}{2}\left((\lambda_1 + \lambda_4)^2 + (\lambda_2 + \lambda_3)^2\right) = 25$ as $\alpha \rightarrow \infty$.

 \section{Numerical Experiments}
\label{sec:numerics}

In this section we present three numerical examples illustrating the effects of the antisymmetric perturbation on the asymptotic variance $\sigma^2_f$.  We will consider target measures defined by Gibbs distributions of the form:
\begin{equation}\label{e:gibbs}
\pi(x) = \frac{1}{Z} e^{-\beta V(x)},
\end{equation}
where $V(x)$ is a given smooth, confining potential with finite unknown normalisation constant  $Z$.  We will also assume that the divergence-free vector field $\gamma(x)$ is given by
\begin{equation}\label{e:drift_decomp}
\gamma(x) = -\alpha J \nabla V(x), \quad J = - J^\top,
\end{equation}
so that the drift of (\ref{e:nonreversible}) will be given by
$$
b(x) = -(\beta I + \alpha J)\nabla V(x).
$$
\\
The symmetric and antisymmetric parts of the generator in $\clr$ are, respectively,
$$
\cS = -\beta \nabla V\cdot \nabla +  \Delta, \quad \cA = J \nabla V \cdot \nabla.
$$
Provided that $V \in H^1(\pi)$ the nullspace of the antisymmetric part of the generator is always nonempty. Indeed, the antisymmetry of $\cA$ implies that $\cA V = 0$.  Thus, for any observable $f$ such that $\pi(f) = 0$ and
$$0 \neq \langle f, V\rangle_\pi,$$
 since $\langle \hat{f}, V\rangle_1 = \langle (-\mathcal{S})^{-1} f, (-\mathcal{S})V \rangle_{\pi} = \langle f, V \rangle_{\pi}$
the conclusion of Theorem \ref{thm:deff_expansion} implies the asymptotic variance $\sigma^2_f(\alpha)$ will converge to a nonzero constant as $\alpha\rightarrow \infty$.

\subsection{Periodic Distribution}
In the first example we consider a two-dimensional target density given by (\ref{e:gibbs}) where the periodic potential $V(x)$ is given by
\begin{equation}
\label{eq:example_pot_1}
  V(x) = \sin(2\pi x_1)\cos(2\pi x_2), \qquad x = (x_1, x_2) \in \T^2,
\end{equation}
and $\beta^{-1} = 0.1$.  Our objective is to calculate the expectation $\pi(f)$ of the observable 
\begin{equation}
\label{eq:periodic_observable}
  f(x) = 1 + 4\sin(4\pi x_1)^2 + 4\cos(4\pi x_2)^2.
\end{equation}
In Figure \ref{fig:periodic1} we plot the asymptotic variance over $\alpha \in [0, 12]$.  The stepsize $\Delta t$ for the Euler--Maruyama discretisation of (\ref{e:nonreversible}) is chosen to be $10^{-3}$.  For each $\alpha$, $M = 10^3$ independent realisations of the Markov process are run, starting from $X_0 = (0, 0)$, each for $T = 10^5$ time units to ensure that the process is close to stationarity.  We observe that increasing $\alpha$ from $0$ to $10$ decreases the asymptotic variance by approximately two orders of magnitude.  In Figure \ref{fig:periodic2} we plot the value of the estimator along with the confidence intervals estimated over $10^3$ independent realisations.  For this particular example, it appears that increasing the magnitude of the nonreversible perturbation does not appear to give rise to any noticeable increase in bias, while it significantly decreases the asymptotic variance,  giving rise to a dramatic increase in performance of the estimator $\pi_T(f)$.

\begin{figure}[ht]
\begin{subfigure}[b]{0.5\textwidth}
                \includegraphics[width=\textwidth]{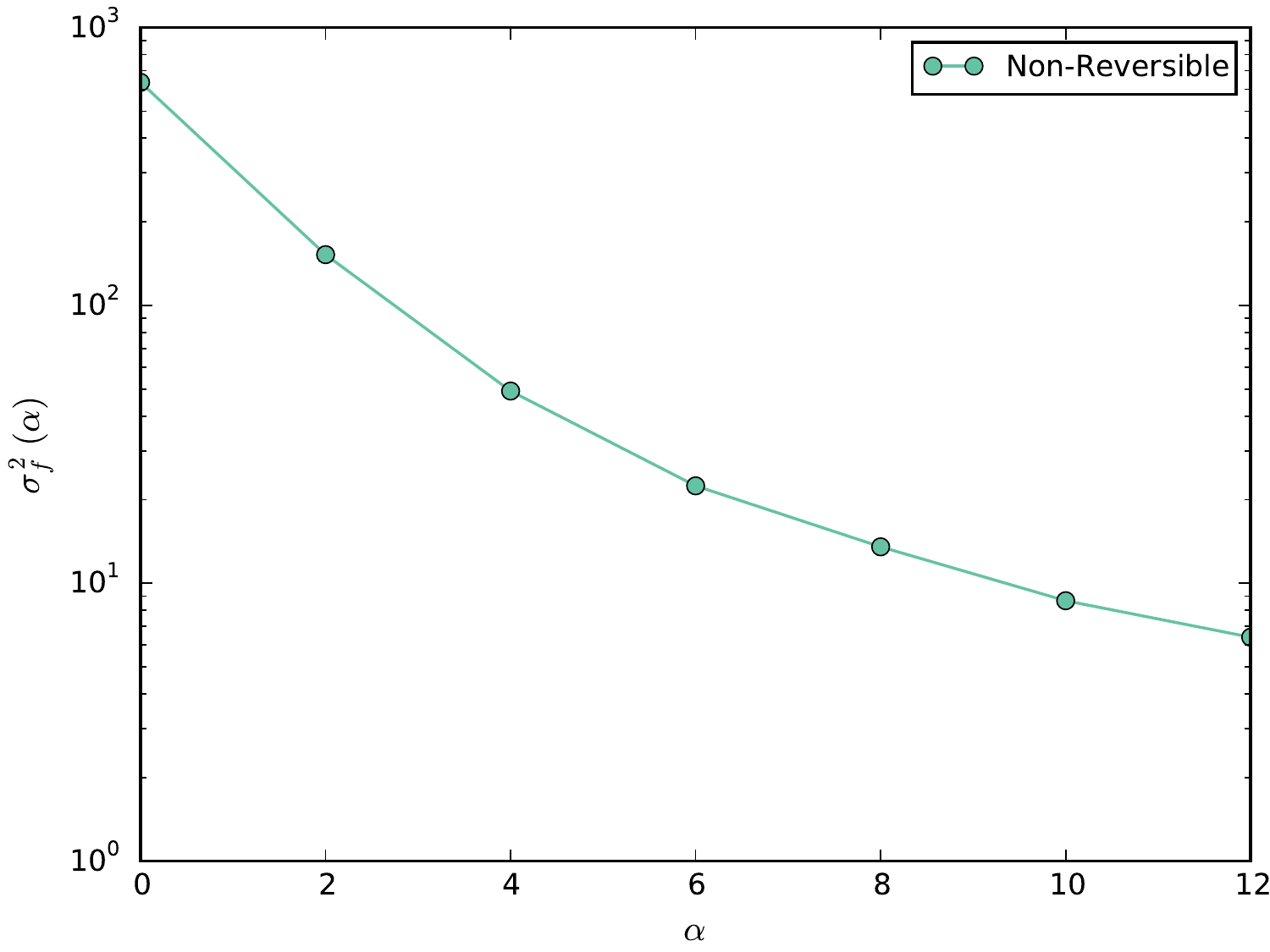}
                \caption{\label{fig:periodic1}Asymptotic Variance $\sigma^2_f(\alpha)$}
        \end{subfigure}%
      ~
      \begin{subfigure}[b]{0.5\textwidth}
                \includegraphics[width=\textwidth]{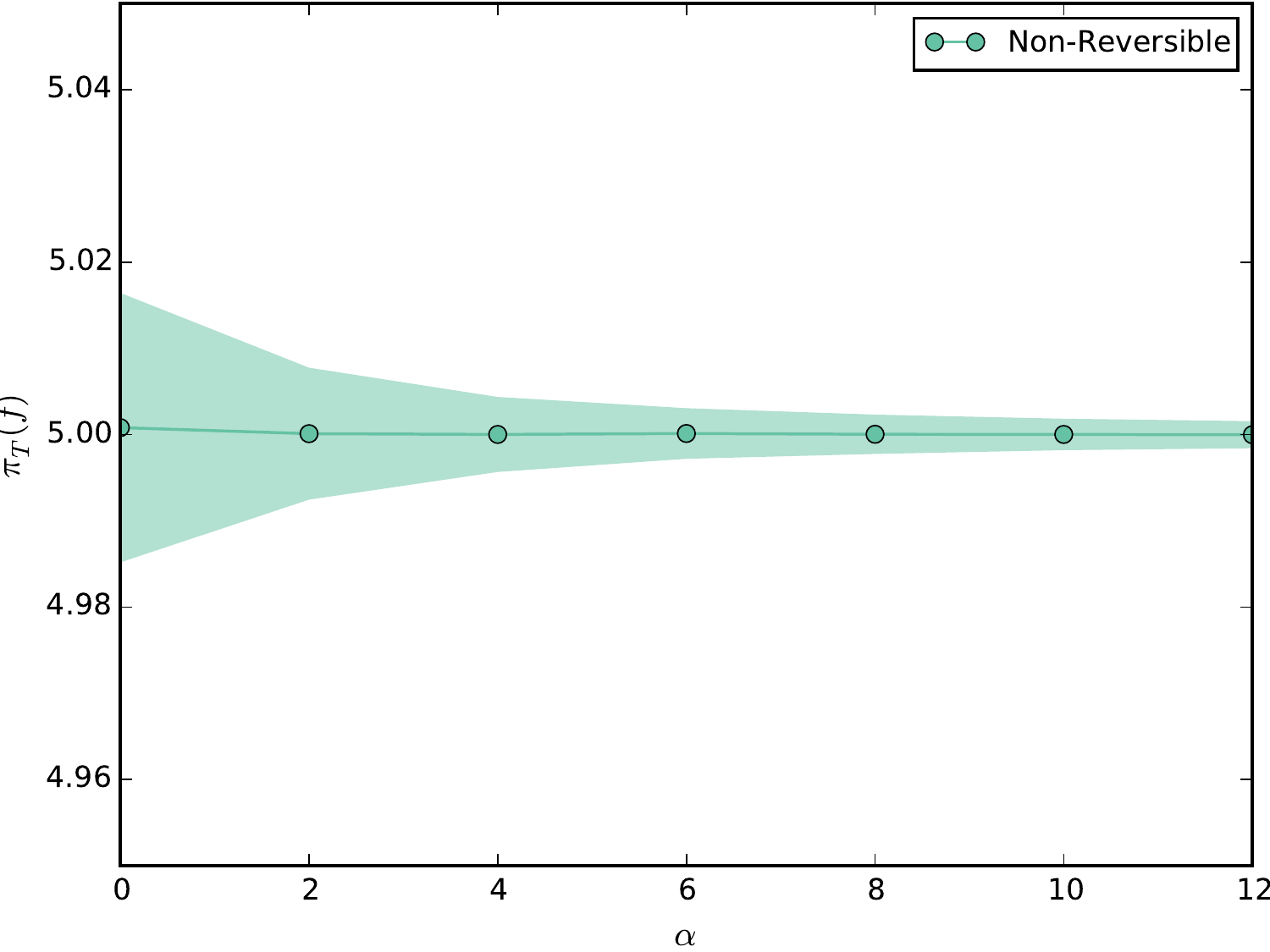}
                \caption{\label{fig:periodic2}Estimator $\pi_T(f)$}
        \end{subfigure}%
\caption{The asymptotic variance and value of the estimator $\pi_T(f)$, where $T = 10^5$, for the target Gibbs distribution with potential defined by (\ref{eq:example_pot_1}) and observable as in (\ref{eq:periodic_observable}).  The plot was generated from $10^3$ independent realisations of an Euler-Maruyama discretisation of the nonreversible diffusion (\ref{e:nonreversible}), with stepsize $10^{-3}$ over $T = 10^{5}$ time-units.  The shaded region in Figure \ref{fig:periodic2} indicates the $95\%$--confidence interval of the estimator for the given $\alpha$.}
\label{fig:periodic_stats}
\end{figure}
%
%
\subsection{Warped Gaussian Distribution}
\label{section:warped}
As a second numerical example we consider computing an observable with respect to a two-dimensional warped Gaussian distribution \cite{haario1999adaptive}, defined by (\ref{e:gibbs}) with
\begin{equation}
\label{eq:warped_potential}
V(x) = \frac{x_1^2}{100} + (x_2 + bx_1^2 - 100b)^2.
\end{equation}
The  parameter $b > 0$ is chosen to be $b = 0.05$.  The potental $V(x)$ is plotted in Figure \ref{fig:warped}.  Our objective is to compute $\pi(f)$ where the observable $f$ is given by:
$$
  f(x) =  x_1^2 + x_2^2.
$$
\begin{figure}[ht]
\begin{center}
\includegraphics[width=0.75\textwidth]{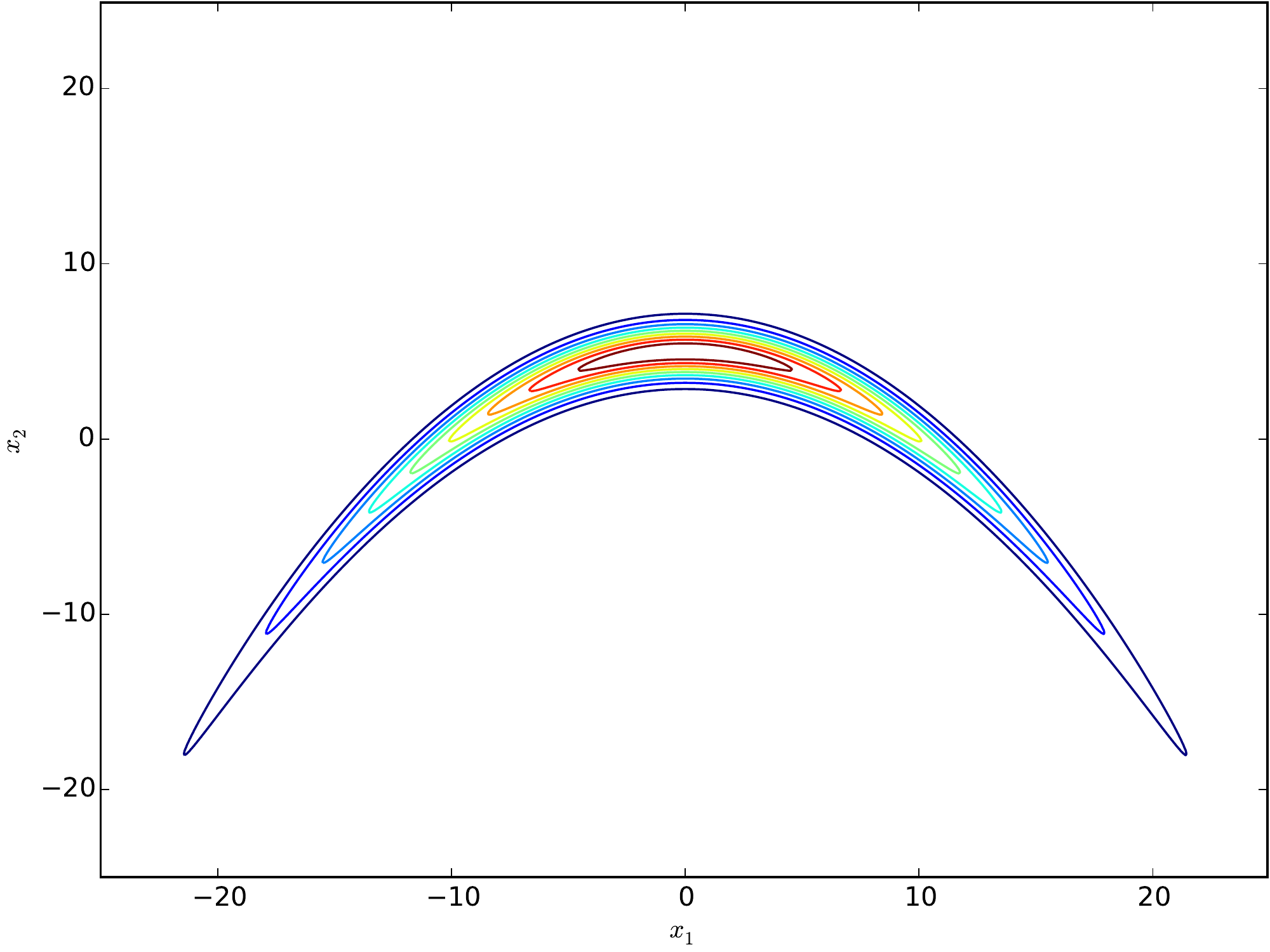}
\end{center}
\caption{Contour plot of the warped Gaussian distribution with potential (\ref{eq:warped_potential}).}
\label{fig:warped} 
\end{figure}
The nonreversible perturbation is chosen to be:
\begin{equation}
\label{eq:nonreversible_drift_warped}
 \gamma(x) = -J\nabla V(x), \quad \mbox{ where } J = \left(\begin{matrix} 0 & 1 \\ -1 & 0 \end{matrix}\right).
\end{equation}
In Figure \ref{fig:warped1} we plot the asymptotic variance of the estimator $\pi_T(f)$, approximated from $10^3$ independent realisations of the process, each run for $10^8$ timesteps of size $10^{-3}$ so that the total time is $T = 10^5$.  Each realisation was started at $(0,0)$.  As expected, we observe a significant decrease in asymptotic variance as the magnitude of the nonreversible perturbation is increased.  Increasing $\alpha$ from $0$ to $10$ gives rise to a decrease in variance by a factor of $80$.    In Figure \ref{fig:warped2} we plot the value of the estimator $\pi_T$ averaged over the $10^3$ realiations, with corresponding confidence intervals.  As $\alpha$ is increased, the bias arising from the discretisation error also increases.  To mitigate this bias one could decrease the stepsize, thus requiring more steps to generate $\pi_T(f)$ or otherwise introduce a Metropolis-Hastings accept-reject step, which will be discussed in  section \ref{subsec:MH}.  The tradeoff between bias and variance is considered in more detail in Section \ref{sec:comp_cost}.

\begin{figure}[ht]
\begin{subfigure}[b]{0.5\textwidth}
                \includegraphics[width=\textwidth]{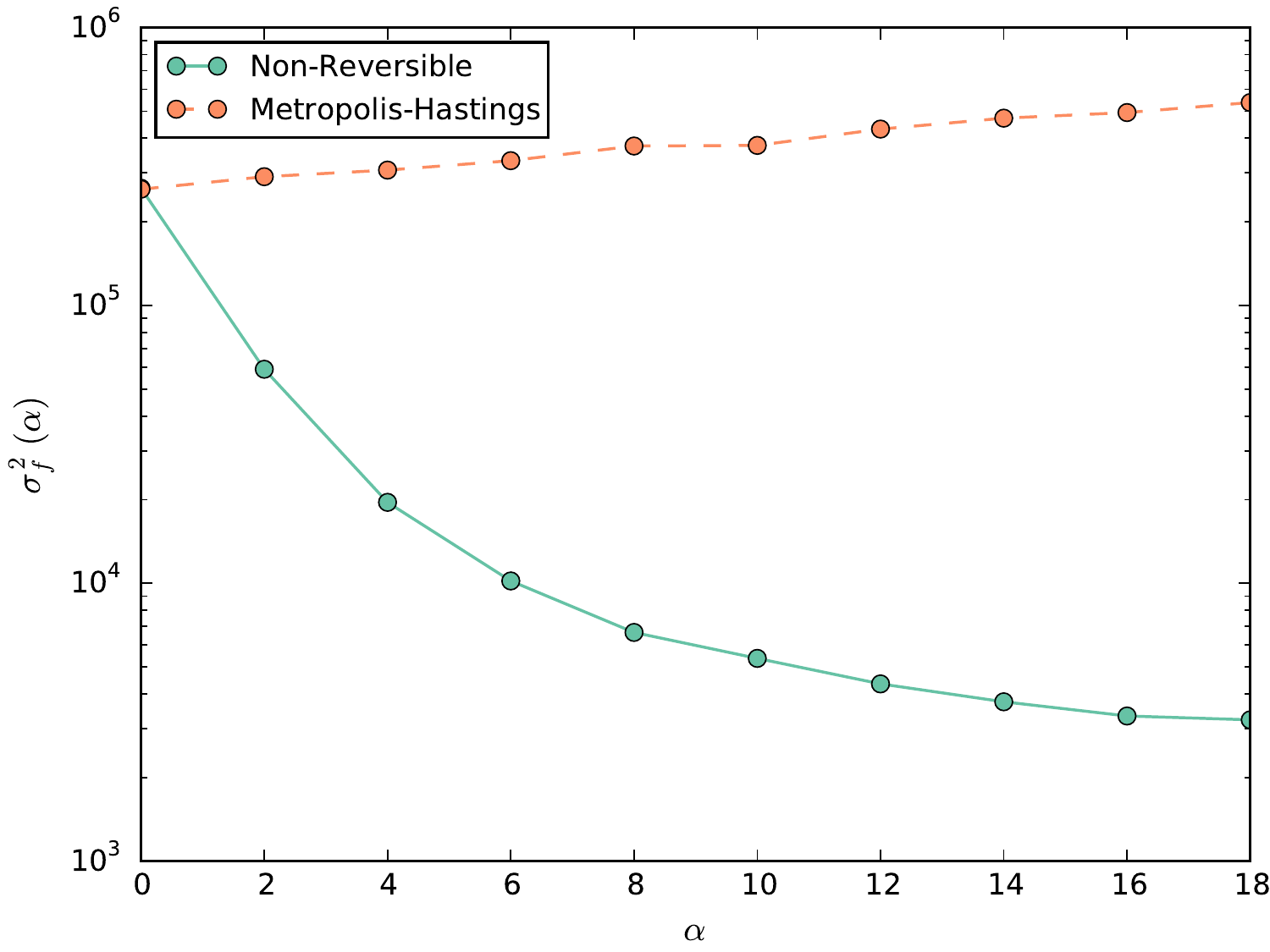}
                \caption{\label{fig:warped1}Asymptotic Variance $\sigma^2_f(\alpha)$}
        \end{subfigure}%
      ~
      \begin{subfigure}[b]{0.5\textwidth}
                \includegraphics[width=\textwidth]{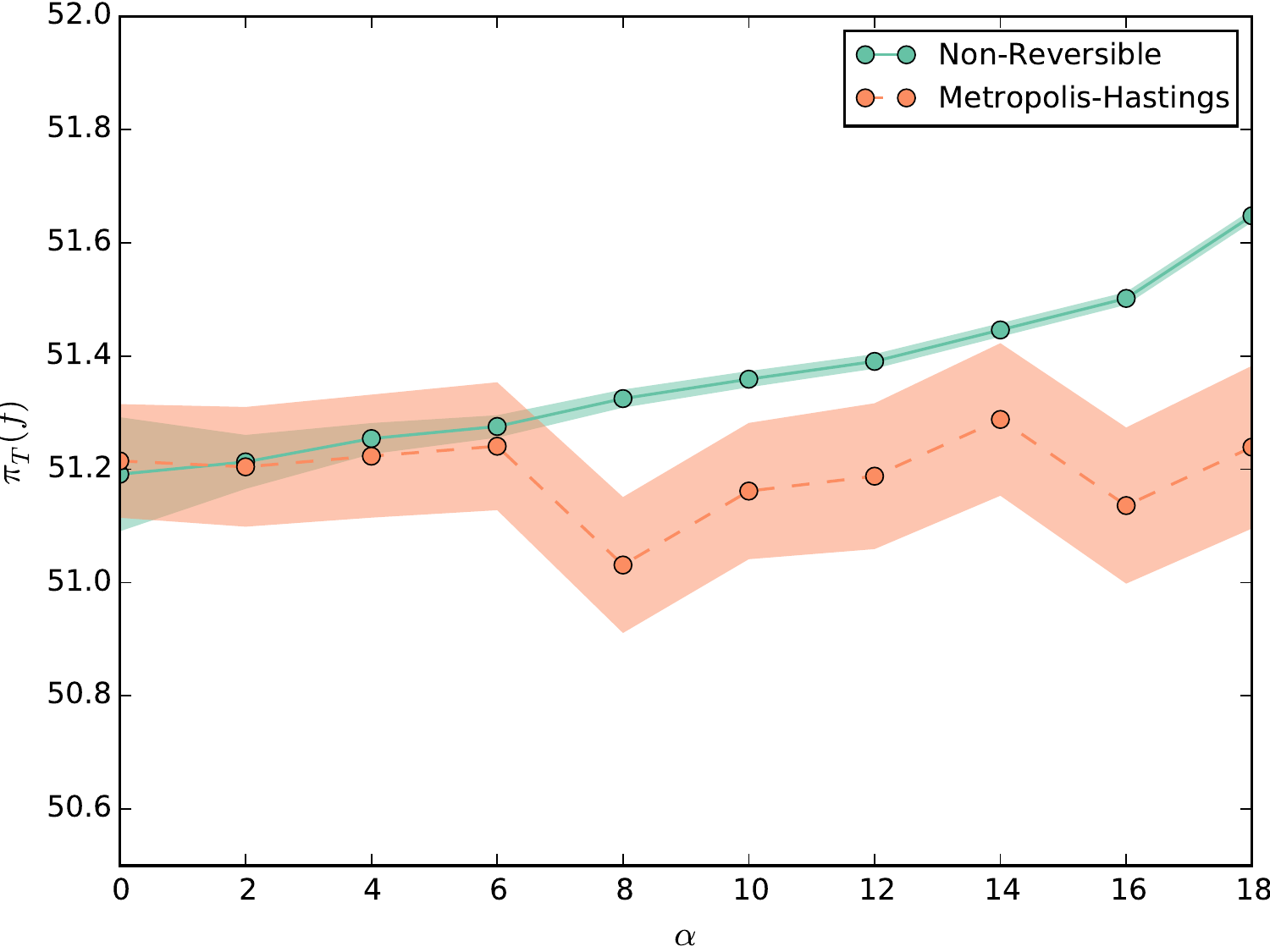}
                \caption{\label{fig:warped2}Estimator $\pi_T(f)$}
        \end{subfigure}%
\caption{Plot of the asymptotic variance $\sigma^2_f(\alpha)$ and the estimator $\pi_T(f)$ for the warped Gaussian distribution $\pi$ defined by (\ref{eq:warped_potential}) and the observable $f(x)= |x|^2$ for different values of $\alpha$. The solid line was generated from $10^3$ independent realisations of an Euler-Maruyama discretisation of the nonreversible diffusion (\ref{e:nonreversible}), with stepsize $10^{-3}$ over $T = 10^{5}$ time-units.  The shaded region in Figure \ref{fig:warped2} indicates the confidence interval of the estimator for the given $\alpha$. The dashed lined denotes the variance and mean when the chain is augmented with a accept/reject step, see Section \ref{subsec:MH}.}
\label{fig:warped_stats}
\end{figure}


\subsection{Introducing a Metropolis--Hastings accept-reject step}
\label{subsec:MH}

When sampling from a distribution $\pi$  it is natural to introduce a Metropolis--Hastings (MH) accept-reject step, rather than use the Markov chain obtained directly from an Euler--Maruyama discretization of the SDE (\ref{e:nonreversible}).  Given a current state $X^{(n)}$, a next state is proposed according to the Euler--Maruyama discretisation of (\ref{e:nonreversible}):
\begin{equation}
\label{eq:mh_proposal}
  \tilde{X} \sim \mathcal{N}\left(X^{(n)} + \Delta t  \nabla \log\pi(X^{(n)}), \Delta t\right), 
\end{equation}
The proposed state is then accepted  (i.e. $X^{(n+1)} := \tilde{X}$) with probability
$$
  r(X^{(n)}, \tilde{X}) = 1 \wedge \frac{\pi(\tilde{X})p(X^{(n)}, \tilde{X})}{\pi(X^{(n)})p(\tilde{X}, X^{(n)})},
$$
where $p(\cdot, x)$ is the proposal density corresponding to (\ref{eq:mh_proposal}).   By introducing this accept-reject mechanism, the resulting chain $X^{(n)}$ is guaranteed to have a stationary distribution which is {exactly} $\pi$, independently of $\Delta t$.  Thus, introducing a Metropolis-Hastings accept-reject step allows for far larger stepsizes to be used,  while still preserving the correct invariant distribution of the chain, eliminating any bias arising from the discretisation of the SDE, making it  beneficial both in terms of computational performance and stability.
\\\\
A natural question is whether an MH chain using a proposal distribution based on the SDE (\ref{e:nonreversible}) with antisymmetric drift will inherit the superior mixing properties of the nonreversible diffusion process.  As the MH algorithm works by enforcing the detailed balance of the chain $X^{(n)}$ with respect to the distribution $\pi$, we expect that any benefits of the antisymmetric drift term will be negated when introducing this accept-reject step.  To test this, we repeat the numerical experiment of Section \ref{section:warped} using the MH algorithm using the Euler discretisation of (\ref{e:nonreversible}) as a proposal scheme,  for various values of $\alpha$.  The effect of introducing this accept-reject step to the nonreversible diffusion is evident from Figure \ref{fig:warped_stats}. While the accept-reject step removes any bias due to discretisation error, as is evident from Figure \ref{fig:warped2}, the asymptotic variance actually increases as $\alpha$ increases.  This is due to the fact that for large $\alpha$, proposals are more likely to be rejected as they are far away from the current state.

\subsection{Dimer in a Solvent}
\label{subsec:dimer}
We now test the effect of adding a nonreversible perturbation to a test model from molecular dynamics, as described in \cite[Section 1.3.2.4]{lelievre2010free}.  We consider a system composed of $N$ particles $P_1, \ldots, P_N$ in a two-dimensional periodic box of side length $L$.  Particles $P_1$ and $P_2$ are assumed to form a dimer pair, in a solvent comprising the particles $P_3, \ldots, P_N$.  The solvent particles interact through a truncated Lennard-Jones potential:
$$
  V_{WCA}(r) = \begin{cases} 4\epsilon\left[\left(\frac{\sigma}{r}\right)^{12} - \left(\frac{\sigma}{r}\right)^6\right] + \epsilon, &\mbox{ if } r \leq r_0, \\ 0 &\mbox{ if } r > r_0,\end{cases}
$$
where $r$ is the distance between two particles, $\epsilon$ and $\sigma$ are two positive parameters, and $r_0 = 2^{\frac{1}{6}}\sigma$.  The interaction potential between the dimer pair is given by a double-well potential
\begin{equation}
  V_S(r) = h\left[1 - \frac{(r - r_0 - w)^2}{w^2}\right]^2,
\end{equation}
where $h$ and $w$ are two positive parameters.  The total energy of the system is given by
\begin{equation}
  V(q) = V_S(|q_1 - q_2|) + \sum_{3 \leq i < j \leq N} V_{WCA}(|q_i - q_j|) + \sum_{i=1,2}\sum_{3 \leq j \leq N} V_{WCA}(|q_i - q_j|),
\end{equation}
with $q = (q_1,\ldots, q_N) \in (L\mathbb{T})^{dN}$, where $q_i$ denotes the position of particle $P_i$.  The potential $V_S$ has two energy minima at $r = r_0$ (corresponding to the compact state), and at $r = r_0 + 2w$ (corresponding to the stretched state).  The energy barrier separating the two states is $h$.  Define $\xi(q)$ to be
$$
  \xi(q) = \frac{|q_1 - q_2| - r_0}{2w}.
$$
This reaction coordinate describes the transition from the compact state, $\xi(q) = 0$ to the stretched state $\xi(q) = 1$.
The standard overdamped reversible dynamics to sample from the stationary distribution $\pi(q)\propto \exp(-\beta V(q))$ is given by
$$
  dq_t = -\nabla V(q_t)\,dt  + \sqrt{2\beta^{-1}}\,dW_t,
$$
where $\beta$ is the inverse temperature.  Our objective is to compute the average reaction coordinate $\mathbb{E}_\pi \xi(q)$. We introduce an antisymmetric drift term as follows
\begin{equation}
\label{eq:dimer_sde_antisymmetric}
  dq_t = -(I + \alpha J)\nabla V(q_t)\,dt + \sqrt{2\beta^{-1}}\,dW_t,
\end{equation}
where $J \in \mathbb{R}^{2N\times 2N}$ is an antisymmetric matrix.  We consider two types of nonreversible perturbations, determined by antisymmetric matrices, $J_1$ and $J_2$.  The first matrix $J_1$ is the block-circulant matrix defined by:
$$
J_1=\left(
\begin{array}{cccccc}
 {O}_{2}    & {I}_{2} & \cdots  & {O}_{2} & -{I}_{2}   \\
 - {I}_{2}    & {O}_{2} &  {I}_{2}    &   & {O}_{2}   \\
  \vdots    & - {I}_{2} & {O}_{2}    & \ddots  & \vdots   \\
 {O}_{2}    &  &  \ddots    &   \ddots &  {I}_{2}   \\
  {I}_{2}    &  {O}_{2} &  \ldots    &   - {I}_{2} & {O}_{2},
\end{array},
\right),
$$
where $I_{2}$ and $O_{2}$ denote the $2\times 2$ identity and zero matrix, respectively.  The second matrix is given by
$$
J_2 = \left(\begin{matrix}R & O_4 & \ldots & O_4 \\ O_4 & O_4  & \ldots & O_4 \\ \vdots & \vdots & \ddots & \vdots\\ O_4  &   O_4       & \ldots & O_4 
\end{matrix}\right),
$$
where $R$ is the following rotation matrix on $\mathbb{R}^{4 \times 4}$:
$$
  R = \left(\begin{matrix}0 & 0 & 1 & 0 \\ 0 & 0 & 0 & 1 \\ -1 & 0 & 0 &0 \\ 0 & -1 & 0 & 0\end{matrix}\right),
$$
and $O_4$ is the $4\times 4$ zero matrix. The effect, in this case, is to apply the antisymmetric transformation only on the first two coordinates, which correspond to the positions of the particles composing the dimer. 
\\\\
Using an Euler-Maruyama discretisation of (\ref{eq:dimer_sde_antisymmetric}),  samples are generated for $N = 16$ particles, with parameter values $\beta = 1.0$, $\sigma = 1.0$, $\epsilon= 1.0$, $h = 1.0$, $w = 0$, and $\alpha \in \lbrace 0, 2, 4, 6, 8, 10 \rbrace$.  For each value of $\alpha$, a single realisation, starting from $\mathbf{0}$ is simulated for $10^{10}$ timesteps of size $\Delta t = 10^{-5}$, so that the total running time is $T = 10^{5}$.  The asymptotic variance is approximated from a single realisation of the process using a batch--means estimator (see \cite[IV.5]{asmussen2007stochastic}).  In Figure \ref{fig:dimer_xi_var} we plot the value of the generated estimator, along with the asymptotic variance for different values of $\alpha$ and for $J_1$ and $J_2$.  The first perturbation provides marginally lower asymptotic variance for this particular observable, giving rise to a $66\%$ decrease, as opposed to a $50\%$ decrease for $J_2$, increasing $\alpha$ from $0$ to $10$.  This decrease in asymptotic variance, comes at the cost of increased computation time due to having to reduce the stepsize to compensate for the discretisation error and numerical instability caused by the large drift.  While this may appear as a negative result, since the antisymmetric drift terms were chosen arbitrarily, no claim is made about optimality.  An important issue is whether the increase in computational cost outweighs the benefits.  This issue will be studied more carefully in Section \ref{sec:comp_cost}.



\begin{figure}[ht]
    \begin{subfigure}[b]{0.5\textwidth}
                \includegraphics[width=\textwidth]{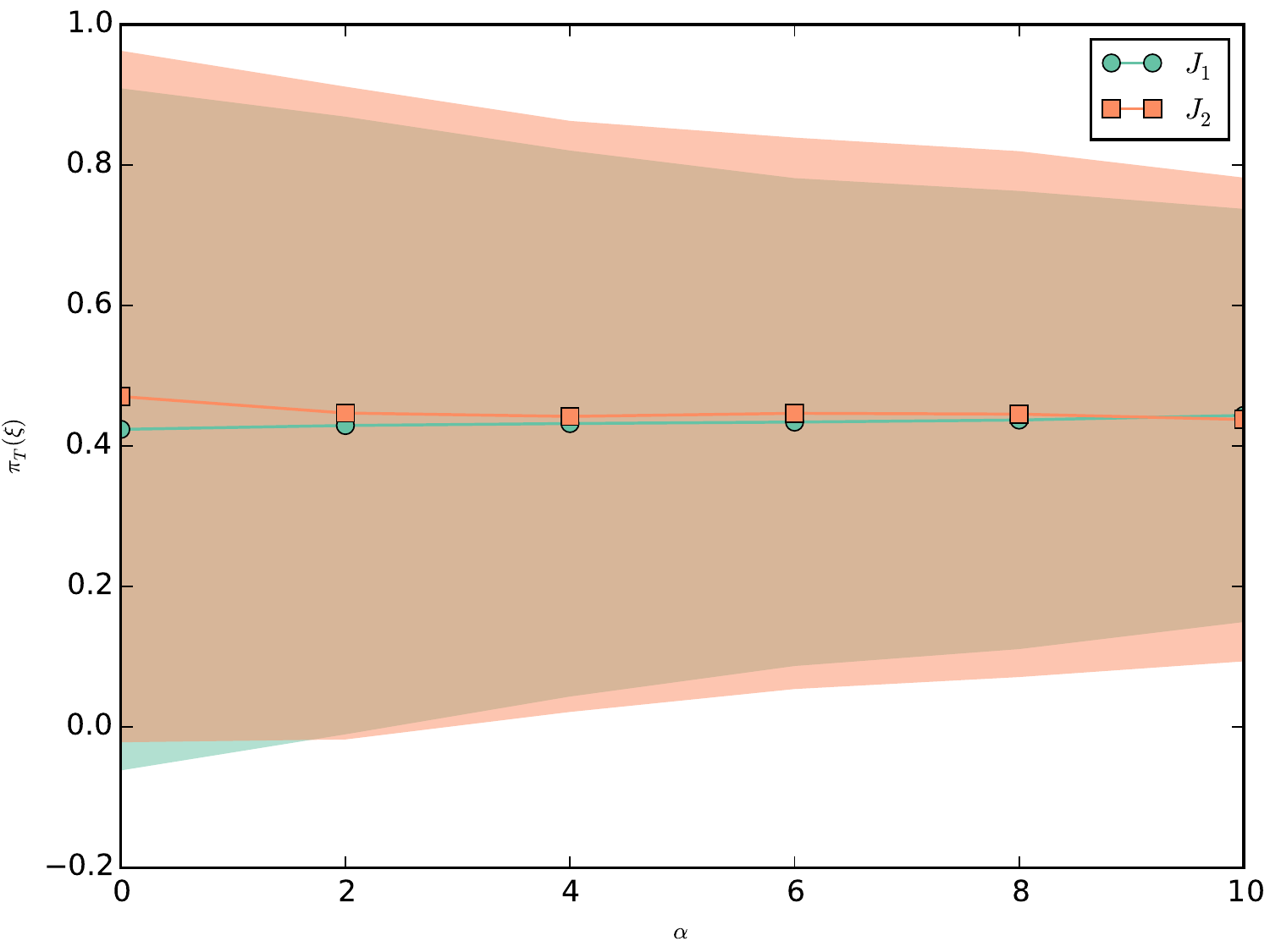}
        \end{subfigure}%
      ~
      \begin{subfigure}[b]{0.5\textwidth}
                \includegraphics[width=\textwidth]{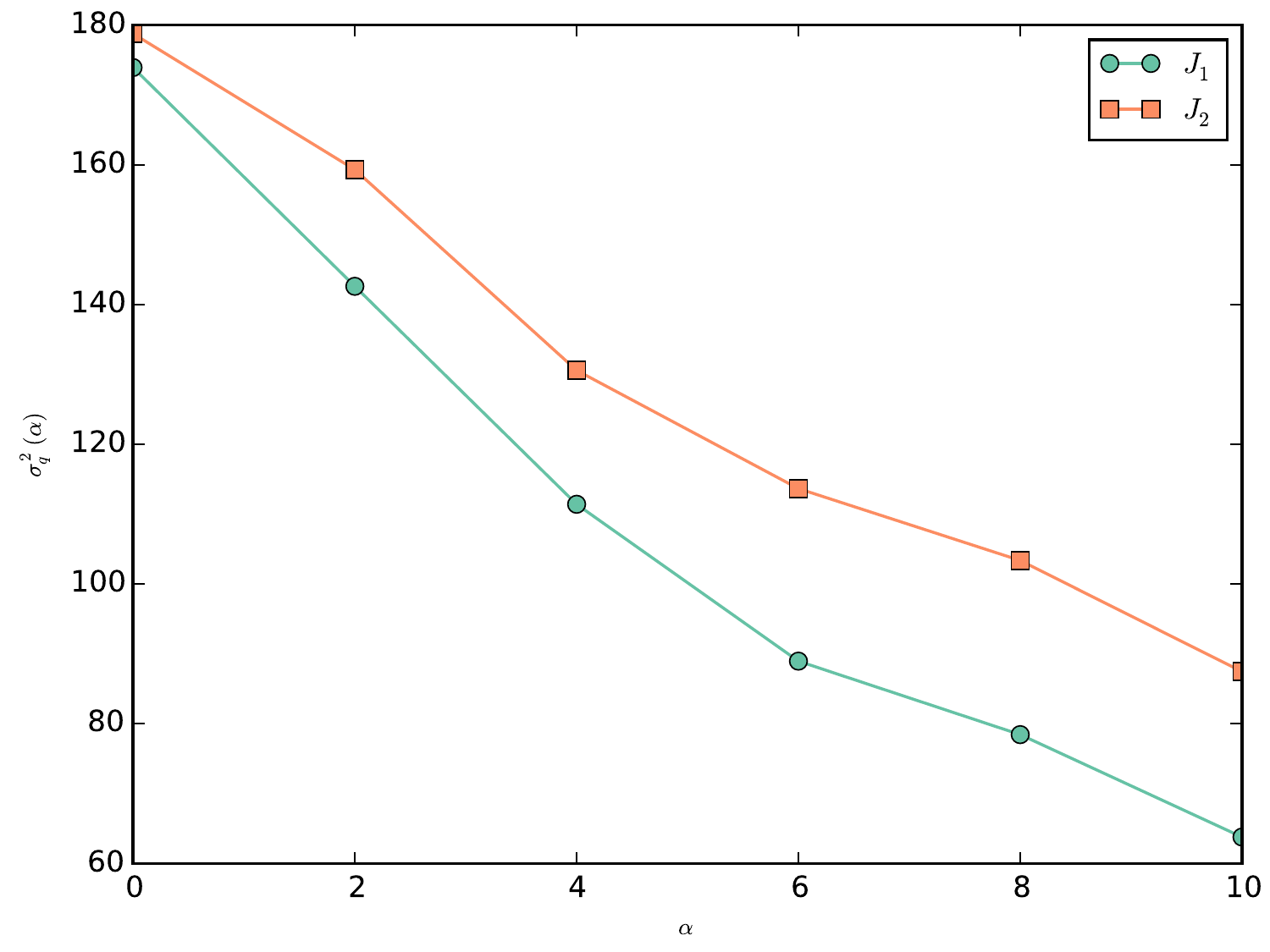}
        \end{subfigure}%
\caption{Value of the estimator $\pi_T(\xi)$ for $T = 10^5$ with corresponding $95\%$--confidence intervals, for the dimer model, generated from a single realisation of (\ref{eq:dimer_sde_antisymmetric}).  The right plot shows the asymptotic variance of the estimator over varying $\alpha$, with antisymmetric drifts defined by $J_1\nabla V(x)$ and $J_2\nabla V(x)$, respectively.}
\label{fig:dimer_xi_var}
\end{figure}
%
%

%
%
\section{The Computational Cost of Nonreversible Langevin Samplers}
\label{sec:comp_cost}
As observed Section \ref{sec:numerics}, while increasing $\alpha$ is guaranteed to decrease the asymptotic variance $\sigma^2_{f}(\alpha)$ for the estimator $\pi_T(f)$,  this will also give rise to an increase in the discretisation error arising from the particular discretisation being used.  Moreover, as $\alpha$ increases, the SDE (\ref{e:nonreversible}) becomes more and more stiff, to the extent that the discretisation becomes numerically unstable unless the stepsize is chosen to be accordingly small.  As a result, any discretisation $\lbrace X^{(n)} \rbrace_{n=1}^N$ will require smaller timesteps to guarantee that the stationary distribution of $X^{(n)}$ is sufficiently close to $\pi(x)$.   This tradeoff between computational cost and asymptotic variance of the estimator must be taken into consideration when comparing reversible to nonreversible diffusions.
\\\\
A rigorous error analysis of the long time average estimator was carried out in \cite{mattingly2010convergence}.   In this paper careful estimates of the mean square error 
$$\mbox{Err}^2_{N,\Delta t}(f) := \mathbb{E}\left|\frac{1}{N}\sum_{n=1}^N f(X^{(n)}) - \pi(f)\right|^2$$
were derived for discretisations of a general overdamped Langevin diffusion on the unit torus $\mathbb{T}^d$, see also \cite{talay1990expansion}.   In particular, in  \cite[Theorem 5.2]{mattingly2010convergence} it is shown that the mean squared error can be bounded as follows
\begin{equation}
\label{eq:numerical_estimator_error}
  Err_{N, \Delta_t}^2[f] \leq C\left(\Delta t^2 + \frac{1}{N\Delta t}\right),
\end{equation}
where $C$ is a positive constant independent of $\Delta t$ and $N$, which depends on the coefficients of the SDE and the observable $f$.  This estimate makes explicit the tradeoff between discretisation error and sampling error. For a fixed computational budget $N$, the right hand side of (\ref{eq:numerical_estimator_error}) is minimized when $\Delta t \propto {(N)^{-\frac{1}{3}}}.$   For an SDE of the form (\ref{e:nonreversible}), we expect that the constant $C$ will increase with $\alpha$.  Identifying the correct scaling of the error with respect to $\alpha$ is an interesting problem that we intend to study.
\\\\
To obtain a clearer idea of the bias variance tradeoff we compute the mean-square error for the Euler-Maruyama discretisation for two particular examples.    In Figure \ref{fig:computational_cost_warped1}, we consider the warped Gaussian distribution defined by (\ref{eq:warped_potential}) and the observable $f(x) = |x|^2$.  A value for $\pi(f)$ is obtained by integrating $\int_{\mathbb{R}^d} f(x)\pi(dx)$ numerically, using a globally adaptive quadrature scheme to obtain an approximation with error less than $10^{-12}$.   In Figure \ref{fig:computational_cost_warped1} we plot the relative mean--squared-error defined by $\left(Err_{N, \Delta t}[f]/\pi(f)\right)^2$ for an Euler-Maruyama discretisation of (\ref{e:nonreversible}), {for timestep $\Delta t$ in the interval $[2^{-5}, 1]$. } The total number of timesteps is kept fixed at $N = 10^6$.  For each value of $\alpha$, the mean square error is approximated over an ensemble  of $256$ independent realisations.  Missing points indicate finite time blowup of the discretized diffusion.  The dashed line denotes the MSE generated from the corresponding MALA sampler, namely an Euler--Maruyama discretisation of the reversible diffusion with an added Metropolis--Hastings accept-reject step.  We note that both the Euler-Maruyama discretisation and the MALA sampler require one evaluation of the gradient term $\nabla \log\pi$ per timestep,  so that comparing an ergodic average obtained from $10^6$ steps of each scheme is fair.
\\\\
A trade-off between discretisation error and variance is evident from Figure \ref{fig:computational_cost_warped1}, and is consistent with the error estimate (\ref{eq:numerical_estimator_error}).  We observe that the nonreversible Langevin sampler outperforms the reversible Langevin sampler by an order of magnitude, with the lowest MSE attained when $\alpha = 10$.  As $\alpha$ is increased beyond this point, the discretisation error is balanced by the decrease in variance, and we observe no further gain in performance.  Nonetheless, despite the fact that the MALA scheme has no bias, the nonreversible sampler, with $\alpha = 5$, outperforms MALA (in terms of MSE) by a significant factor of $8.8$.
\\\\
We repeat this numerical experiment for the target distribution $\pi$ given by a standard Gaussian distribution in $\mathbb{R}^d$ and observable $f(x) = x_2 + x_3$.  In this case $\pi(f)$ is exacty $0$.  We use the linear diffusion specified by (\ref{eq:sde_linear}) where the antisymmetric matrix $J$ is given in (\ref{eq:J_linear}), which is optimal for this observable.  We plot the (absolute) MSE for the estimator $\pi_T(f)$ in Figure \ref{fig:computational_cost_ou1}.  While the smallest MSE is attained by the nonreversible Langevin sampler, when $\alpha = 25$, the increase in performance is only marginal.  This is due to the fact that increasingly smaller timesteps must be taken to ensure that the EM approximation does not blow  up.  Indeed, the $\alpha = 25$ sampler would not converge to a finite value for $\Delta t$ greater than $10^{-3}$, while the reversible sampler ($\alpha = 0$) and the MALA scheme were accurate even for timesteps of order $1$.  
\\\\
From both examples it is clear that managing the numerical stability and discretisation error of the skew-symmetric drift term is essential for any practical implementation of the nonreversible sampling scheme.  This suggests that using higher order and/or more stable numerical integrators to compute long time averages would be beneficial to eliminate the bias arising from discretisation, as well as permit the use of larger timesteps.  Naturally, such schemes would require multiple evaluations of $\nabla \log \pi$ at each step, thus it is possible that the additional computational cost offsets any performance gain.   In the remainder of this section, we will perform the same numerical experiments using an integrator based on a Strang splitting \cite{strang1968construction,leimkuhler2004simulating} of the stochastic reversible and deterministic nonreversible dynamics. The reversible part will be simulated using a standard MALA scheme and the nonreversible flow using an appropriate higher--order  integrator.  Indeed, denote by $\Phi_{r, t}$ the evolution of the reversible SDE (\ref{eq:sde1}) from time $0$ to $t$, and let $\Phi_{n, t}(x)$ denote the flow map corresponding to the ODE:
$$
  \dot{z}(t) = \gamma(z(t)), \quad z(0) = x.
$$
We shall consider an integrator based on the following map from time $t$ to $t + \Delta t$:
$$
  \Psi_{\Delta t} = \Phi_{r, \Delta t/2}\circ \Phi_{n, \Delta t}\circ \Phi_{r, \Delta t/2}.
$$
For this implementation, we approximate $\Phi_{r, \Delta t}(x)$ using a single step of a MALA scheme with proposal based on (\ref{eq:sde1}) with stepsize $\Delta t$. The nonreversible flow $\Phi_{n, \Delta t}$ is approximated using a fourth-order Runge--Kutta method.  
\\\\
We leave the justification and analysis of this scheme as the goal of future work, and in this paper simply use it to compute a long time average approximation to $\pi(f)$ and compare the MSE with that of a corresponding reversible MALA scheme.   To obtain a fair comparison between the results obtained by MALA and the splitting scheme, we note that while a careful implementation of MALA requires only one evaluation of $\nabla \log\pi$ per timestep, the splitting scheme requires six evaluations of $\nabla\log\pi$ per timestep (naively a single timestep would require $2$ evaluations for each reversible substep and $4$ evaluations for the nonreversible substep, however we can reuse two evaluations of  $\nabla\log \pi$ between the steps). Thus, we shall compare the MSE obtained from trajectories of $10^6$ timesteps of the nonreversible sampler with $6\cdot 10^6$ timesteps of the corresponding MALA scheme,  for stepsizes ranging from $10^{-5}$ to $1$.  The results for the warped Gaussian distribution in $\mathbb{R}^2$ and standard Gaussian in $\mathbb{R}^3$ are plotted in Figures \ref{fig:computational_cost_warped2} and \ref{fig:computational_cost_ou2}, respectively.   {Note that we omit the $\alpha = 0$ case since, in this case, the splitting scheme reduces to standard the MALA scheme}.  We observe that with this splitting scheme, the nonreversible sampler outperforms MALA by a factor of $13$ for the warped Gaussian model, and by a factor of $20$ for the standard Gaussian model.  The benefits of the splitting scheme appear to be twofold:  firstly the  integrator is more stable,  in both models, the long time simulation of $X_t^\gamma$ did not blow up, even for large values of $\alpha$, and for $\Delta t = 0.1$. Moreover, compared to the corresponding Euler--Maruyama discretisation, the MSE is consistently an order of magnitude less.
\\\\
While this splitting scheme is only a first step into properly investigating appropriate  integrators for nonreversible Langevin schemes, the above numerical experiments demonstrate clearly that there is a significant benefit in doing so, which motivates future investigation.


\begin{figure}[tp]
        \centering
        \begin{subfigure}[b]{0.5\textwidth}
                \includegraphics[width=\textwidth]{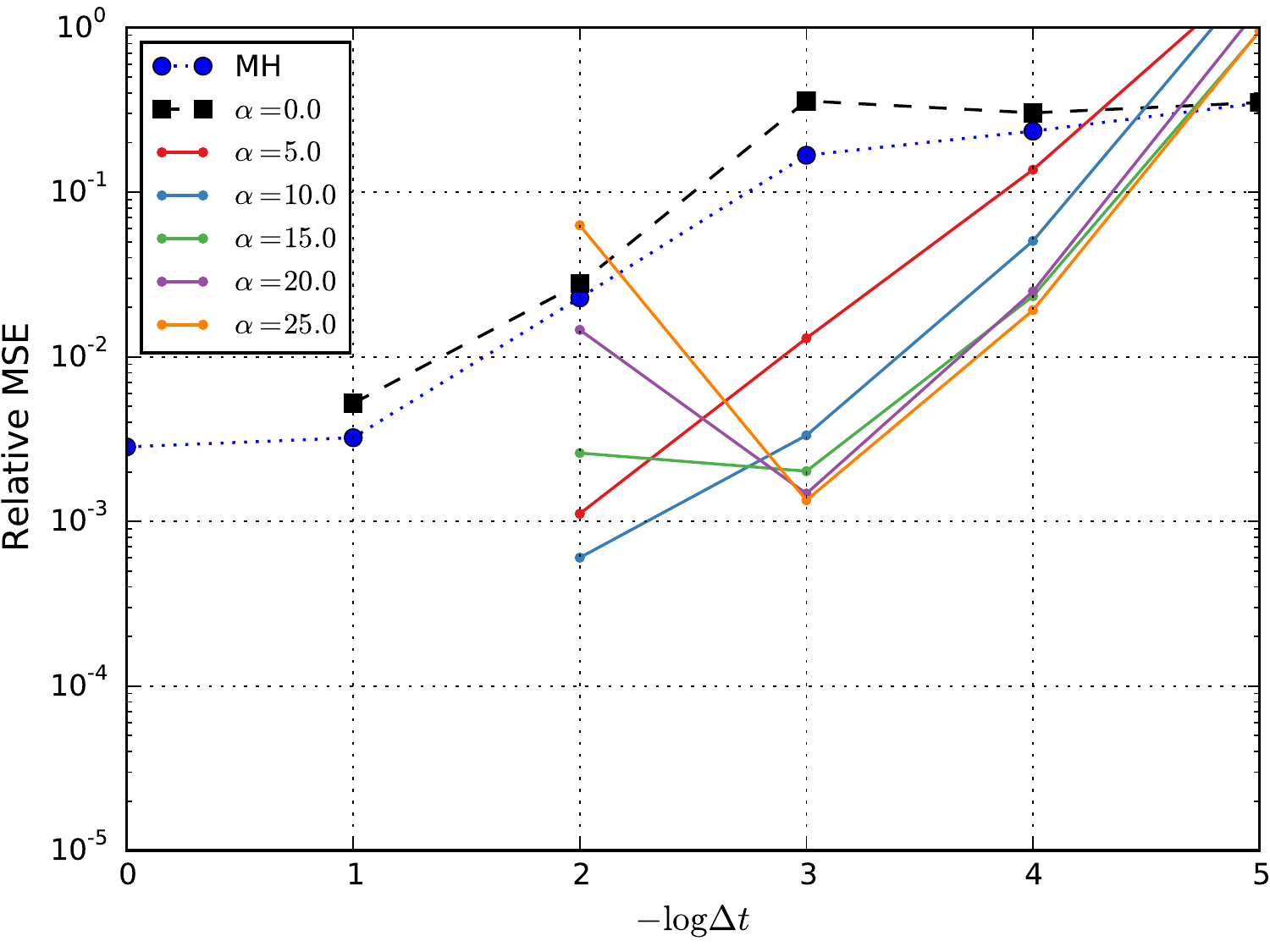}
                \caption{\label{fig:computational_cost_warped1} Relative MSE of the Euler-Maruyama discretisation.}
        \end{subfigure}%
        \begin{subfigure}[b]{0.5\textwidth}
                \includegraphics[width=\textwidth]{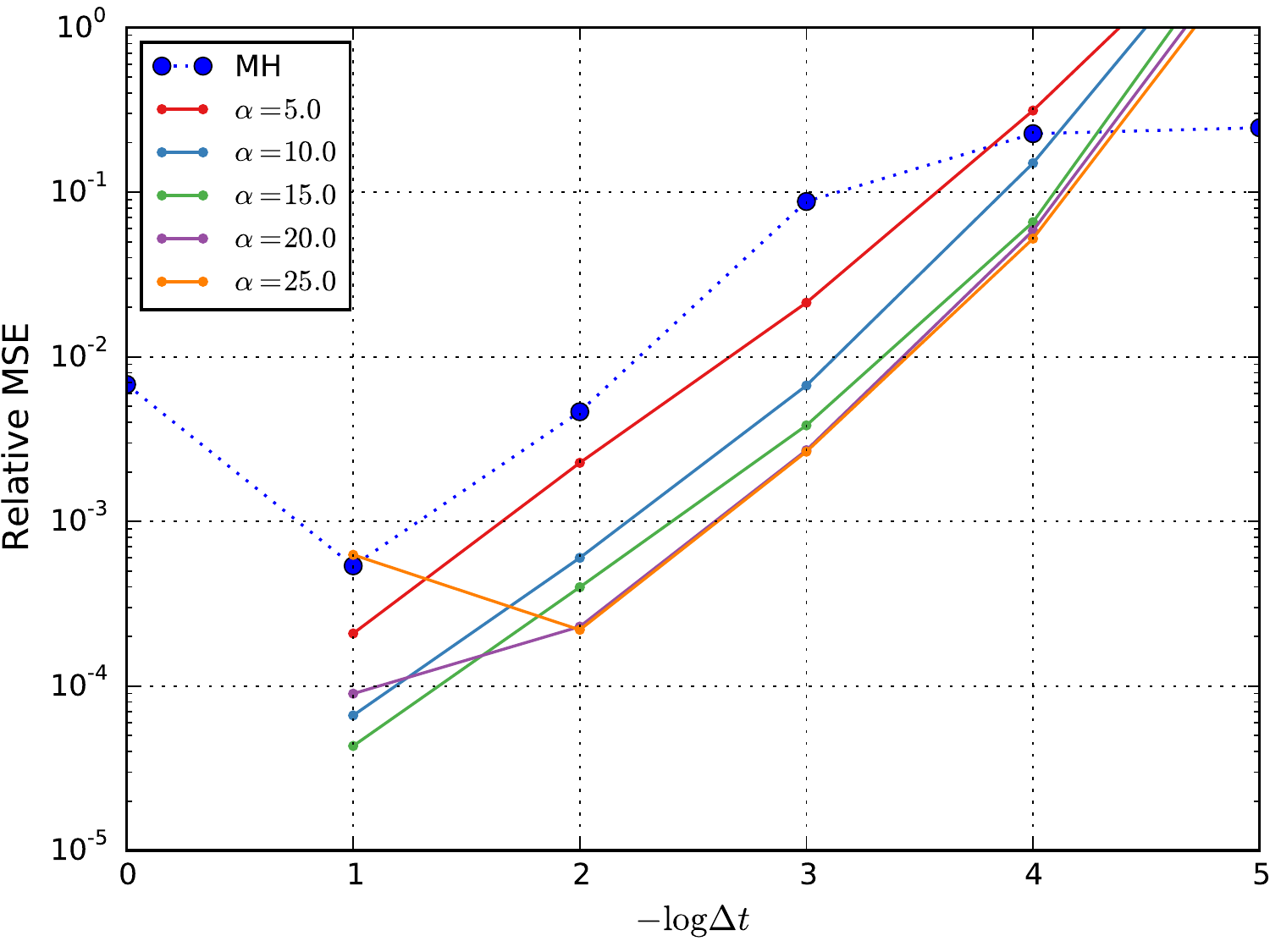}
                \caption{\label{fig:computational_cost_warped2} Relative MSE of the splitting scheme.}
        \end{subfigure}
\caption{Plot of the relative mean-square error as a function of the timestep size $\Delta t$ for the observable $f(x) = |x|^2$ of a warped Gaussian distribution defined by (\ref{eq:warped_potential}), with a fixed computational budget of $10^6$ evaluations of $\nabla\log\pi$ for Figure \ref{fig:computational_cost_warped1} and $6\cdot 10^6$ for Figure \ref{fig:computational_cost_warped2}. The dotted line denotes the relative MSE for the MALA scheme.}
\label{fig:computational_cost_warped} 
\end{figure}


\begin{figure}[tp]
        \centering
        \begin{subfigure}[b]{0.5\textwidth}
                \includegraphics[width=\textwidth]{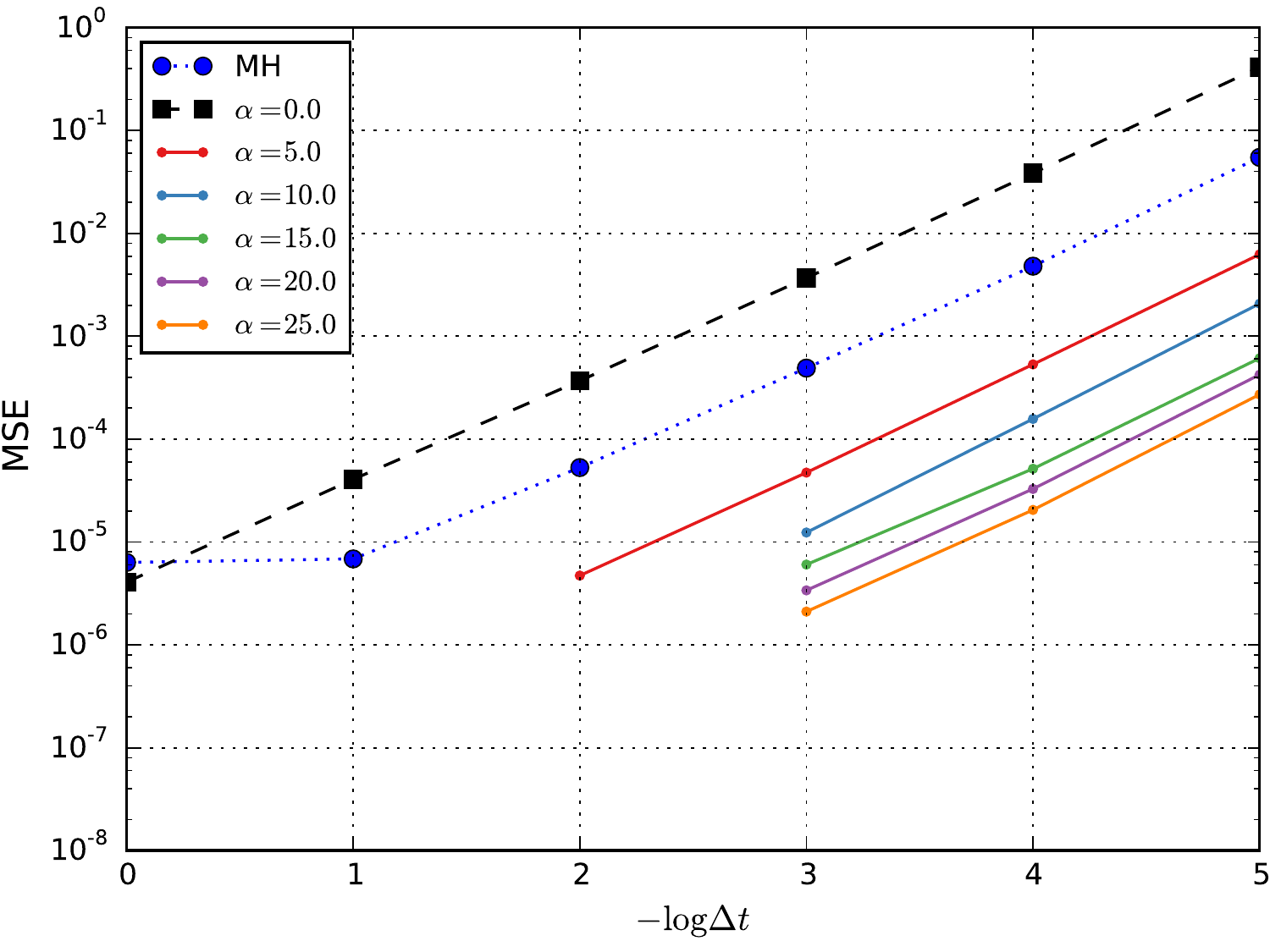}
                \caption{\label{fig:computational_cost_ou1} MSE for the Euler-Maruyama discretisation.}
        \end{subfigure}%
        \begin{subfigure}[b]{0.5\textwidth}
                \includegraphics[width=\textwidth]{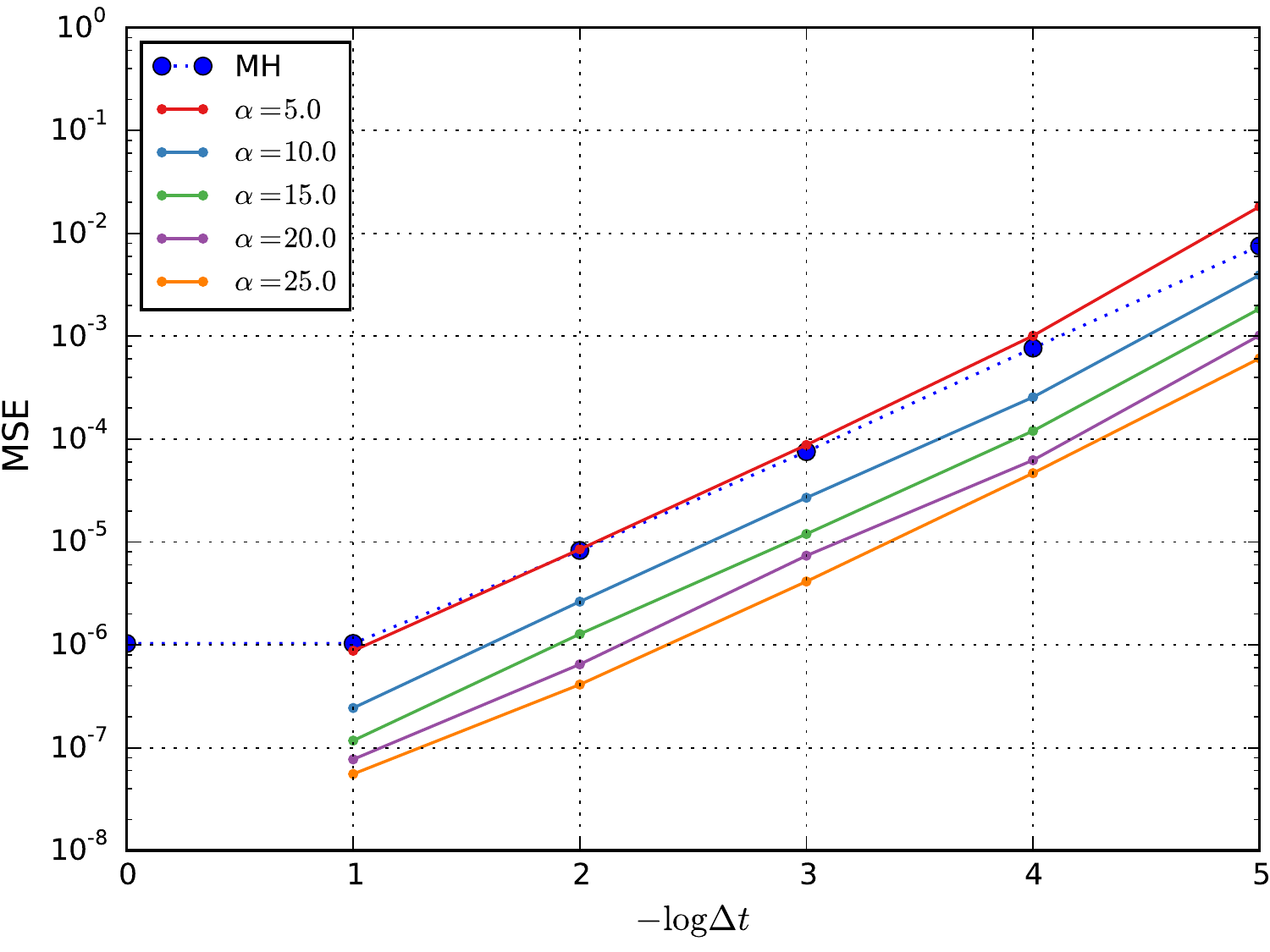}
                \caption{\label{fig:computational_cost_ou2} MSE for the splitting scheme.}
        \end{subfigure}
\caption{Plot of the mean-square error as a function of the timestep size $\Delta t$ for a linear observable of a Gaussian distribution, so that $X^\gamma_t$ is a linear diffusion process.  A fixed computational budget of $10^6$ evaluations of $\nabla\log\pi$ is imposed in Figure \ref{fig:computational_cost_ou1}, and $6\cdot 10^{6}$ in Figure \ref{fig:computational_cost_ou2}.  The dotted line depicts the corresponding error for the Metropolized sampler. }
\label{fig:computational_cost_ou} 
\end{figure}

%
\section{Conclusions and Further Work}
\label{sec:conclusions}

In this paper we have presented a detailed analytical and numerical study of the effect of nonreversible perturbations to Langevin samplers. In particular, we have focused on the effect on the asymptotic variance of adding a nonreversible drift to the overdamped Langevin dynamics. Our theoretical analysis,  presented for diffusions with periodic coefficients and for diffusions with linear drift for which a complete analytical study can be performed, and our numerical investigations on toy models  clearly show that a judicious choice of the nonreversible drift can lead to a substantial reduction in the asymptotic variance.  On the other hand, as observed from the dimer model example in Section \ref{subsec:dimer}, an arbitrary choice of nonreversible drift will not always give rise to significant improvement. We have also presented a careful study of the computational cost of the algorithm based on a nonreversible Langevin sampler, in which the competing effects of reducing the asymptotic  variance and of increasing the stiffness due to the addition of a nonreversible drift are monitored. The main conclusions that can be drawn from our numerical experiments is that a nonreversible Langevin sampler with close--to--the--optimal choice of the nonreversible drift significantly outperforms the (reversible) Metropolis-Hastings sampler.
\\\\
There are many open problems that need to be studied further:

\begin{enumerate}

\item The effect of using degenerate, hypoelliptic diffusions for sampling from a given distribution.

\item Combining the use of nonreversible Langevin samplers with standard variance reduction techniques such as the zero variance reduction MCMC methodology~\cite{dellaportas2012control}.

\item Optimizing Langevin samplers within the class of reversible diffusions.

\item The development of nonreversible Metropolis--Hastings algorithms based on the above techniques, possibly related to approach described in \cite{bierkens2014non}.

\item The development and analysis of numerical schemes specifically designed to simulate nonreversible Langevin diffusions.
\end{enumerate}

All these topics are currently under investigation.

\section*{Acknowledgments}
The work of TL is supported by the European Research Council under the European Union's Seventh Framework Program/ ERC Grant Agreement number 614492.  GP thanks Ch. Doering for useful discussions and comments. The research of AD is supported by the EPSRC under grant No. EP/J009636/1. GP is partially supported by the EPSRC under grants No. EP/J009636/1, EP/L024926/1, EP/L020564/1 and EP/L025159/1. TL would like to thank J. Roussel for pointing out some mistakes in a preliminary version of the manuscript.


\bibliographystyle{plain}
\end{document}